\documentclass[12pt]{amsart}

\usepackage{verbatim}
\usepackage{amssymb}
\usepackage{upref}
\usepackage[all]{xy}
\usepackage[usenames,dvipsnames]{color}
\usepackage{url}
\usepackage[normalem]{ulem}

\allowdisplaybreaks

\newtheorem{thm}{Theorem}[section]
\newtheorem*{thm*}{Theorem}
\newtheorem{lem}[thm]{Lemma}
\newtheorem{cor}[thm]{Corollary}
\newtheorem{prop}[thm]{Proposition}

\theoremstyle{definition} %definition
\newtheorem{defn}[thm]{Definition}
\newtheorem{q}[thm]{Question}

\newtheorem{ex}[thm]{Example}
\newtheorem{notn}[thm]{Notation}

\newtheorem*{notn*}{Notation}

\newtheorem*{hyp*}{Hypothesis}

\theoremstyle{remark} %remark
\newtheorem{rem}[thm]{Remark}
\newtheorem*{rem*}{Remark}

\numberwithin{equation}{section}

\newcommand{\secref}[1]{Section~\textup{\ref{#1}}}

\newcommand{\thmref}[1]{Theorem~\textup{\ref{#1}}}
\newcommand{\corref}[1]{Corollary~\textup{\ref{#1}}}
\newcommand{\lemref}[1]{Lemma~\textup{\ref{#1}}}
\newcommand{\propref}[1]{Proposition~\textup{\ref{#1}}}

\newcommand{\defnref}[1]{Definition~\textup{\ref{#1}}}
\newcommand{\remref}[1]{Remark~\textup{\ref{#1}}}
\newcommand{\qref}[1]{Question~\textup{\ref{#1}}}

\newcommand{\exref}[1]{Example~\textup{\ref{#1}}}

\newcommand{\appxref}[1]{Appendix~\textup{\ref{#1}}}

\newcommand{\midtext}[1]{\quad\text{#1}\quad}
\newcommand{\righttext}[1]{\quad\text{#1 }}

\renewcommand{\)}{\textup)}

\newcommand{\KK}{\mathcal K}

\newcommand{\LL}{\mathcal L}
\newcommand{\FF}{\mathcal F}

\newcommand{\TT}{\mathcal T}

\renewcommand{\epsilon}{\varepsilon}

\DeclareMathOperator{\ad}{Ad}

\DeclareMathOperator{\dashind}{\!-Ind}

\DeclareMathOperator{\prim}{Prim}

\DeclareMathOperator*{\spn}{span}
\DeclareMathOperator*{\clspn}{\overline{\spn}}

\newcommand{\id}{\text{\textup{id}}}

\newcommand{\<}{\langle}
\renewcommand{\>}{\rangle}

\newcommand{\inv}{^{-1}}
\newcommand{\iso}{\overset{\cong}{\longrightarrow}}

\renewcommand{\bar}{\overline}
\newcommand{\what}{\widehat}
\newcommand{\wilde}{\widetilde}

\newcommand{\smtx}[1]{\left(\begin{smallmatrix} #1
\end{smallmatrix}\right)}

\newcommand{\ann}{^\perp}
\newcommand{\pann}{{}\ann}

\newcommand{\cs}{\mathbf{C}^*}
\newcommand{\ac}{\mathbf{Act}}
\newcommand{\co}{\mathbf{Coact}}
\newcommand{\CP}{\textup{CP}}
\newcommand{\cp}{\CP}

\newcommand{\ize}{\text{$E$-ize}}

\begin{document}
\title{Coaction functors}

\author[Kaliszewski]{S. Kaliszewski}
\address{School of Mathematical and Statistical Sciences
\\Arizona State University
\\Tempe, Arizona 85287}
\email{kaliszewski@asu.edu}

\author[Landstad]{Magnus~B. Landstad}
\address{Department of Mathematical Sciences\\
Norwegian University of Science and Technology\\
NO-7491 Trondheim, Norway}
\email{magnusla@math.ntnu.no}

\author[Quigg]{John Quigg}
\address{School of Mathematical and Statistical Sciences
\\Arizona State University
\\Tempe, Arizona 85287}
\email{quigg@asu.edu}



%\date{June 15, 2015}
\date{\today}

\begin{abstract}
A certain type of functor on a category of coactions of a locally compact group on $C^*$-algebras is introduced and studied.
These functors are intended to help in the study of the crossed-product functors that have been recently introduced in relation to the Baum-Connes conjecture.
The most important coaction functors are the ones induced by large ideals of the Fourier-Stieltjes algebra.
It is left as an open problem whether the ``minimal exact and Morita compatible crossed-product functor'' is induced by a large ideal.
\end{abstract}
\maketitle

\section{Introduction}
\label{intro}

In \cite{bgwexact}, with an eye toward expanding the class of locally compact groups $G$ for which the Baum-Connes conjecture holds, the authors study ``crossed-product functors'' that take an action of $G$ on a $C^*$-algebra and produce an ``exotic crossed product'' between the full and reduced ones, in a functorial manner.

In \cite{graded}, inspired by \cite{BrownGuentner}, we studied certain quotients of $C^*(G)$ 
that lie ``above'' $C^*_r(G)$ --- namely those that carry a quotient coaction. We characterized these intermediate (which we now call ``large'') quotients as those for which the annihilator $E$, in the Fourier-Stieltjes algebra $B(G)$, of the kernel of the quotient map is a $G$-invariant weak* closed ideal containing the reduced Fourier-Stieltjes algebra $B_r(G)$ (which we now call ``large ideals'' of $B(G)$). We went on to show how, if $\alpha$ is an action of $G$ on a $C^*$-algebra $B$, large ideals $E$ induce 
exotic crossed products
$B\rtimes_{\alpha,E}G$ intermediate between the full and reduced crossed products $B\rtimes_\alpha G$ and $B\rtimes_{\alpha,r} G$. One of the reasons this interested us is the possibility of ``$E$-crossed-product duality'' for a coaction $\delta$ of $G$ on a $C^*$-algebra $A$: namely, that the canonical surjection $\Phi:A\rtimes_\delta G\rtimes_{\what\delta} G\to A\otimes \KK(L^2(G))$ descends to an isomorphism $A\rtimes_\delta G\rtimes_{\what\delta,E} G\cong A\otimes\KK$. Crossed-product duality $A\rtimes_\delta G\rtimes_{\what\delta,r} G\cong A\otimes\KK$ for normal coactions and $A\rtimes_\delta G\rtimes_{\what\delta} G\cong A\otimes\KK$ for maximal coactions are the extreme cases with $E=B_r(G)$ and $B(G)$, respectively. We (rashly) conjectured that every coaction satisfies $E$-crossed-product duality for some $E$, and moreover that the dual coaction on every $E$-crossed product $B\rtimes_{\alpha,E} G$ satisfies $E$-crossed-product duality.

In \cite{BusEch} Buss and Echterhoff disproved the first of the above conjectures, and in \cite{exotic} we proved the second conjecture (also proved independently in \cite{BusEch}).
(Note: in \cite[Introduction]{exotic} we wrote
``We originally wondered whether every coaction satisfies $E$-crossed product duality for some $E$. In \cite[Conjecture~6.12]{graded} we even conjectured that this would be true for dual coactions.'' This is slightly inaccurate --- \cite[Conjecture~6.14]{graded} concerns dual coactions, while \cite[Conjecture~6.12]{graded} says
``Every coaction satisfies $E$-crossed-product duality for some $E$.'')

In \cite[Section~3]{exotic} we showed that every large ideal $E$ of $B(G)$ induces a transformation $(A,\delta)\mapsto (A^E,\delta^E)$ of $G$-coactions, where $A^E=A/A_E$ and
$A_E=\ker (\id\otimes q_E)\circ\delta$, and where in turn $q_E:C^*(G)\to C^*_E(G):=C^*(G)/\pann E$ is the quotient map.

In this paper we further study this assignment $(A,\delta)\mapsto (A^E,\delta^E)$.
When $(A,\delta)=(B\rtimes_\alpha G,\what\alpha)$, the composition
\[
(B,\alpha)\mapsto (B\rtimes_\alpha G,\what\alpha)
\mapsto (B\rtimes_{\alpha,E} G,\what\alpha^E)
\]
was shown to be functorial in \cite[Corollary~6.5]{BusEch};
here we show that $(A,\delta)\mapsto (A^E,\delta^E)$ is functorial, giving an alternate proof of the Buss-Echterhoff result.

In fact, we study more general functors on the category of coactions of $G$,
of which the functors induced by large ideals of $B(G)$ are special cases.
We are most interested in the connection with the crossed-product functors of \cite{bgwexact}.
In particular, we introduce a 
``minimal exact and Morita compatible" coaction functor.
When this functor is composed with the full-crossed-product functor for actions,
the result is a crossed-product functor in the sense of \cite{bgwexact}.
We briefly discuss various possibilities for how these functors are related:
for example, is the composition mentioned in the preceding sentence equal to the minimal exact and Morita compatible crossed product functor of \cite{bgwexact}?
Also, is the greatest lower bound of the coaction functors defined by large ideals itself defined by a large ideal?
These are just two among others that arise naturally from these considerations.
Unfortunately, at this early stage we have more questions than answers.

After a short section on preliminaries, in \secref{categories} we define the categories we will use for our functors.
In numerous previous papers, we have used ``nondegenerate categories'' of $C^*$-algebras and their equivariant counterparts. But these categories are inappropriate for the current paper, primarily due to our need for short exact sequences. Rather, here we must use ``classical'' categories, where the homomorphisms go between the $C^*$-algebras themselves, not into multiplier algebras. In order to avail ourselves of tools that have been developed for the equivariant nondegenerate categories, we include a brief summary of how the basic theory works for the classical categories.
Interestingly, the crossed products are the same in both versions of the categories (see Corollaries~\ref{same coaction product} and \ref{same product}).

In \secref{sec:coaction functor} we define \emph{coaction functors}, which are a special type of functor on the classical category of coactions.
Composing such a coaction functor with the full-crossed-product functor on actions, we get crossed-product functors in the sense of Baum, Guentner, and Willett (\cite{bgwexact}); it remains an open problem whether every such crossed-product functor is of this form.
Maximalization and normalization are examples of coaction functors, but there are lots more --- for example, the functors induced by large ideals of the Fourier-Stieltjes algebra (see \secref{large}).
In \secref{sec:coaction functor}
we also define a partial ordering on coaction functors, and prove in \thmref{glb} that the class of coaction functors is complete in the sense that every nonempty collection of them has a greatest lower bound.
We also introduce the general notions of \emph{exact} or \emph{Morita compatible} coaction functors, and prove in \thmref{glb exact} that they are preserved by greatest lower bounds.
We show in \propref{compose} that our partial order, as well as our exactness and Morita compatibility, are consistent with those of \cite{bgwexact}.

To help prepare for the study of coaction functors associated to large ideals,
in \secref{decreasing} we introduce \emph{decreasing coaction functors},
and show how Morita compatibility takes a particularly simple form for these functors in \propref{decreasing morita}.

In \secref{E coaction functor} we study the coaction functors $\tau_E$ induced by large ideals $E$ of $B(G)$.
Perhaps interestingly, maximalization is not among these functors.
We show that these functors $\tau_E$ are decreasing in \propref{E coaction functor}, and how the test for exactness simplifies significantly for them in \propref{exact E functor}. Moreover, $\tau_E$ is automatically Morita compatible (see \propref{morita}).
Composing maximalization followed by $\tau_E$, we get a related functor that we call \emph{$E$-ization}. We show that these functors are also Morita compatible in \thmref{E morita}. Although $E$-ization and $\tau_E$ have similar properties, they are not naturally isomorphic functors (see \remref{mE}).
The outputs of $E$-ization are precisely the coactions we call \emph{$E$-coactions},
namely those for which \emph{$E$-crossed-product duality} holds \cite[Theorem~4.6]{exotic} (see also \cite[Theorem~5.1]{BusEch}). \thmref{equivalence} shows that $\tau_E$ gives an equivalence of maximal coactions with $E$-coactions.
We close \secref{E coaction functor} with some open problems that mainly concern the application of the coaction functors $\tau_E$ to the theory of \cite{bgwexact}.

Finally, \appxref{module lemmas} supplies a few tools that show how some properties of coactions can be more easily handled using the associated $B(G)$-module structure.

We thank the referee for comments that significantly improved our paper.

\section{Preliminaries}\label{prelim}

We refer to \cite[Appendix~A]{enchilada}, \cite{maximal} for background material on coactions of locally compact groups on $C^*$-algebras, and \cite[Chapters~1--2]{enchilada} for 
imprimitivity bimodules
and their linking algebras.
Throughout, $G$ will denote a locally compact group, and $A,B,C,\dots$ will denote $C^*$-algebras.

Recall from \cite[Definition~1.14]{enchilada} that the \emph{multiplier bimodule} of an $A-B$ imprimitivity bimodule $X$ is defined as $M(X)=\LL_B(B,X)$, where $B$ is regarded as a Hilbert module over itself in the canonical way. Also recall \cite[Corollary~1.13]{enchilada} that $M(X)$ becomes an $M(A)-M(B)$ correspondence in a natural way.
The \emph{linking algebra} of an $A-B$ imprimitivity bimodule $X$ is $L(X)=\smtx{A&X\\\wilde X&B}$,
where $\wilde X$ is the \emph{dual} $B-A$ imprimitivity bimodule.
$A$, $B$, and $X$ are recovered from $L(X)$ via the \emph{corner projections} $p=\smtx{1&0\\0&0},q=\smtx{0&0\\0&1}\in M(L(X))$.
The multiplier algebra of $L(X)$ decomposes as $M(L(X))=\smtx{M(A)&M(X)\\M(\wilde X)&M(B)}$.
We usually omit the lower left corner of the linking algebra, writing $L(X)=\smtx{A&X\\{*}&B}$, since it takes care of itself.
Also recall from \cite[Lemma~1.52]{enchilada} (see also \cite[Remark~(2) on page 307]{er:multipliers}) that nondegenerate homomorphisms of imprimitivity bimodules correspond bijectively to nondegenerate homomorphisms of their linking algebras.

For an action $(A,\alpha)$ of $G$,
we use the following notation for the (full) crossed product $A\rtimes_\alpha G$:
\begin{itemize}
\item $i_A=i_A^\alpha:A\to M(A\rtimes_\alpha G)$
and $i_G=i_G^\alpha:G\to M(A\rtimes_\alpha G)$
comprise the universal covariant homomorphism $(i_A,i_G)$.

\item $\what\alpha$ is the dual coaction on $A\rtimes_\alpha G$.
\end{itemize}

On the other hand, for the reduced crossed product $A\rtimes_{\alpha,r} G$ we use:
\begin{itemize}
\item $\Lambda:A\rtimes_\alpha G\to A\rtimes_{\alpha,r} G$ is the regular representation.

\item $i_A^r=i_A^{\alpha,r}=\Lambda\circ i_A$
and $i_G^r=i_G^{\alpha,r}=\Lambda\circ i_G$
are the canonical maps into $M(A\rtimes_{\alpha,r} G)$.

\item $\what\alpha^n$ is the dual coaction on $A\rtimes_{\alpha,r} G$.
\end{itemize}

We will need to work extensively with morphisms between coactions,
in particular (but certainly not only) with maximalization and normalization.
In the literature, the notation for these maps has not yet stabilized.
Recall that a coaction $(A,\delta)$ is called \emph{normal} if the canonical surjection
$\Phi:A\rtimes_\delta G\rtimes_{\what\delta} G\to A\otimes \KK(L^2(G))$
factors through an isomorphism of the reduced crossed product
$\Phi:A\rtimes_\delta G\rtimes_{\what\delta,r} G\to A\otimes \KK(L^2(G))$,
and \emph{maximal} if $\Phi$ itself is an isomorphism.
One convention is,
for a coaction $(A,\delta)$ of $G$,
to write
\[
q^m_A:(A^m,\delta^m)\to (A,\delta)
\]
for a maximalization,
and
\[
q^n_A:(A,\delta)\to (A^n,\delta^n)
\]
for a normalization.
We will use this convention for maximalization, but we will need the letter ``$q$'' for other similar purposes, and it would be confusing to keep using it for normalization.
Instead, we will use
\[
\Lambda=\Lambda_A:(A,\delta)\to (A^n,\delta^n)
\]
for normalization --- this is supposed to remind us that for crossed products by actions the regular representation
\[
\Lambda:(A\rtimes_\alpha G,\what\alpha) \to (A\rtimes_{\alpha,r} G,\what\alpha^n)
\]\
is a normalization.

\subsection*{$B(G)$-modules}

Every coaction $(A,\delta)$ of $G$ induces $B(G)$-module structures on both $A$ and $A^*$:
for $f\in B(G)$ define
\begin{align*}
&f\cdot a=(\id\otimes f)\circ\delta(a)\righttext{for}a\in A\\
&(\omega\cdot f)(a)=\omega(f\cdot a)\righttext{for}\omega\in A^*,a\in A.
\end{align*}
Many properties of coactions can be handled using these module structures rather than the coactions themselves.
For example (see \appxref{module lemmas}),
letting $(A,\delta)$ and $(B,\epsilon)$ be coactions of $G$:
\begin{enumerate}
\item
A homomorphism
$\phi:A\to B$ is $\delta-\epsilon$ equivariant,
meaning $\epsilon\circ\phi=\bar{\phi\otimes\id}\circ\delta$,
if and only if
\[
\phi(f\cdot a)=f\cdot \phi(a)\righttext{for all}f\in B(G),a\in A.
\]

\item
An ideal $I$ of $A$ is \emph{weakly} $\delta$-invariant,
meaning $I\subset \ker\bar{q\otimes\id}\circ\delta$, where $q:A\to A/I$ is the quotient map,
if and only if
\[
B(G)\cdot I\subset I,
\]
because
the proof of \cite[Lemma~3.11]{graded} shows that
\[
\ker(q\otimes\id)\circ\delta=\{a\in A:B(G)\cdot a\subset I\}.
\]
\end{enumerate}
If $I$ is a weakly $\delta$-invariant ideal of $A$, then in fact $I=\ker(q\otimes\id)\circ\delta$, and the quotient map $q$ is $\delta-\delta^I$ equivariant for a unique coaction $\delta^I$ on $A/I$, which we call the \emph{quotient coaction}.
Since the slice map $\id\otimes f:M(A\otimes C^*(G))\to M(A)$ is strictly continuous \cite[Lemma~1.5]{lprs},
the $B(G)$-module structure extends to $M(A)$, and moreover $m\mapsto f\cdot m$ is strictly continuous on $M(A)$ for every $f\in B(G)$.

\subsection*{Short exact sequences}
Several times we will need the following elementary lemma.

\begin{lem}\label{nine}
Let
\[
\xymatrix{
&0\ar[d]&0\ar[d]&0\ar[d]
\\
0\ar[r]&A_1\ar[r]^{\phi_1}\ar[d]_{\iota_A}
&B_1\ar[r]^{\psi_1}\ar[d]_{\iota_B}
&C_1\ar[r]\ar[d]_{\iota_C}
&0
\\
0\ar[r]&A_2\ar[r]^{\phi_2}\ar[d]_{\pi_A}
&B_2\ar[r]^{\psi_2}\ar[d]_{\pi_B}
&C_2\ar[r]\ar[d]_{\pi_C}
&0
\\
0\ar[r]&A_3\ar[r]^{\phi_3}\ar[d]
&B_3\ar[r]^{\psi_3}\ar[d]
&C_3\ar[r]\ar[d]
&0
\\
&0&0&0
}
\]
be a commutative diagram of $C^*$-algebras,
where the columns and the middle row are exact.
Suppose that the $\iota$'s are inclusions of ideals and the $\pi$'s are quotient maps.
Then the bottom \(interesting\) row is exact if and only if both
\begin{equation}\label{exact 1}
\phi_2(A_1)=\phi_2(A_2)\cap B_1
\end{equation}
and
\begin{equation}\label{exact 2}
\phi_2(A_2)+B_1\supset \psi_2\inv(C_1).
\end{equation}
\end{lem}

\begin{proof}
Since $\psi_3\circ\pi_B=\pi_C\circ\psi_2$
and $\psi_B$ and $\phi_2$ are both surjective, 
$\psi_3$ is surjective,
so
the bottom row is automatically exact at $C_3$.

Thus, the only items to consider are exactness of 
the bottom row at $A_3$ and $B_3$,
i.e., whether
$\phi_3$ is injective and $\phi_3(A_3)=\ker\psi_3$.

$\phi_3$ is injective if and only if
$\ker \pi_A=\ker \pi_B\circ\phi_2$,
which, since $\phi_2$ is injective, is equivalent to
\eqref{exact 1}.

Since $\psi_2\circ\phi_2=0$ and $\pi_A$ is surjective, $\psi_3\circ\phi_3=0$,
so $\phi_3(A_3)\subset\ker\psi_3$ automatically.
Since 
$\pi_B$ is surjective,
$\phi_3(A_3)\supset\ker \psi_3$
if and only if
\[
\pi_B\inv(\phi_3(A_3)\supset \pi_B\inv(\ker\psi_3).
\]
Since
$\pi_B\inv(\phi_3(A_3))$ consists of all $b\in B_2$
for which
\[
\pi_B(a)\in \phi_3(A_3)=\phi_3(\pi_A(A_2))=\pi_B(\phi_2(A_2)),
\]
equivalently
for which
\[
b\in \phi_2(A_2)+B_1,
\]
we see that
\[
\pi_B\inv(\phi_3(A_3))=\phi_2(A_2)+B_1.
\]
On the other hand,
\[
\pi_B\inv(\ker\psi_3)=\ker\psi_3\circ \pi_B=\ker \pi_C\circ\psi_2
=(\psi_2)\inv(C_1).
\]
Thus, the bottom row is exact at $B_3$ if and only if
\eqref{exact 2} holds.
\end{proof}

\begin{rem}
In the above lemma, we were interested in characterizing exactness of the bottom (interesting) row of the diagram.
\cite[Lemma~3.5]{bgwexact} does this
in terms of subsets of the spectrum $\what{B_2}$,
which could just as well be done with subsets of
$\prim B_2$,
but we instead did it directly in terms of ideals of $B_2$.
Note that, although the $\iota$'s were inclusion maps of ideals and the $\pi$'s were the associated quotient maps, for technical reasons we did \emph{not} make the analogous assumptions regarding the middle row.

There is a standard characterization from homological algebra,
namely that the bottom row is exact if and only if the top row is ---
this is sometimes called the nine lemma, and is an easy consequence of the snake lemma.
However, this doesn't seem to lead to a simplification of the above proof.
\end{rem}

\section{The categories and functors}\label{categories}

We want to study coaction functors.
Among other things, we want to apply the theory we've developed in \cite{graded, exotic}
concerning large ideals $E$ of $B(G)$.
On the other hand, it is important to us in this paper for 
our theory to be consistent with the crossed-product functors of \cite{bgwexact}.
In particular, we want to be able to apply our coaction functors to short exact sequences.

But now a subtlety arises: some of us working in noncommutative duality for $C^*$-dynamical systems have grown accustomed to doing every thing in the ``nondegenerate'' categories, where the morphisms are nondegenerate homomorphisms into multiplier algebras (possibly preserving some extra structure).
But the maps in a short exact sequence
\[
\xymatrix@C-10pt{
0\ar[r]& I\ar[r]^\phi& A\ar[r]^\psi& B\ar[r]&0
}
\]
are not of this type, most importantly $\phi$.
So, we must replace the nondegenerate category by something else.
We can't just allow arbitrary homomorphisms into multiplier algebras, because they wouldn't be composable.
We can't require ``extendible homomorphisms'' into multiplier algebras, because the inclusion of an ideal won't typically have that property.
Thus, it seems we need to use the ``classical category'' of homomorphisms between the $C^*$-algebras, not into multiplier algebras.
This is what \cite{bgwexact} uses, so presumably our best chance of seamlessly connecting with their work is to do likewise.

Since most of the existing categorical theory of coactions uses nondegenerate categories, it behooves us to establish the basic theory we need in the context of the classical categories, which we do below.

One drawback to this is that the covariant homomorphisms and crossed products can't be constructed using morphisms from the classical $C^*$-category --- so, it seems we have to abandon some of the appealing features of the nondegenerate category.

\begin{defn}
In the \emph{classical category $\cs$ of $C^*$-algebras}, a morphism $\phi:A\to B$ is a *-homomorphism from $A$ to $B$ in the usual sense (no multipliers).
\end{defn}

\begin{defn}
In the \emph{classical category $\co$ of coactions}, a morphism $\phi:(A,\delta)\to (B,\epsilon)$ is a morphism $\phi:A\to B$ in $\cs$ such that
the diagram
\[
\xymatrix@C+10pt{
A \ar[r]^-\delta \ar[d]_\phi
&\wilde M(A\otimes C^*(G)) \ar[d]^{\bar{\phi\otimes\id}}
\\
B \ar[r]_-\epsilon
&\wilde M(B\otimes C^*(G))
}
\]
commutes,
and we call $\phi$ a \emph{$\delta-\epsilon$ equivariant} homomorphism.
\end{defn}

To make sense of the above commuting diagram, recall that for any $C^*$-algebra $C$,
\[
\wilde M(A\otimes C)=\{m\in M(A\otimes C);m(1\otimes C)\cup (1\otimes C)m\subset A\otimes C\},
\]
and that for any homomorphism $\phi:A\to B$ there is a canonical extension to a homomorphism
\[
\bar{\phi\otimes\id}:\wilde M(A\otimes C)\to \wilde M(B\otimes C),
\]
by \cite[Proposition~A.6]{enchilada}.
It is completely routine to verify that $\cs$ and $\co$ are categories, i.e., there are identity morphisms and there is an associative composition.

\begin{rem}
Thus, a coaction is not itself a morphism in the classical category; this will cause no trouble.
\end{rem}

To work in the classical category of coactions, we need to be just a little bit careful with covariant homomorphisms and crossed products.
We write $w_G$ for the unitary element of
$M(C_0(G)\otimes C^*(G))=C_b(G,M^\beta(C^*(G)))$
defined by $w_G(s)=s$, where we have identified $G$ with its canonical image in $M(C^*(G))$,
and where the superscript $\beta$ means that we use the strict topology on $M(C^*(G))$.

\begin{defn}\label{cov def}
A \emph{degenerate covariant homomorphism} of a coaction $(A,\delta)$ to a $C^*$-algebra $B$ is a pair $(\pi,\mu)$,
where 
$\pi:A\to M(B)$ and $\mu:C_0(G)\to M(B)$ are homomorphisms such that
$\mu$ is nondegenerate
and 
the diagram
\[
\xymatrix@C+60pt{
A \ar[r]^-\delta \ar[d]_\pi
&\wilde M(A\otimes C^*(G)) \ar[d]^{\bar{\pi\otimes\id}}
\\
M(B) \ar[r]_-{\ad(\mu\otimes\id)(w_G)\circ (\cdot\otimes 1)}
&M(B\otimes C^*(G))
}
\]
commutes,
where the bottom arrow is the map
$b\mapsto \ad(\mu\otimes\id)(w_G)(b\otimes 1)$.
If $\pi:A\to M(B)$ happens to be nondegenerate, we sometimes refer to $(\pi,\mu)$ as a \emph{nondegenerate covariant homomorphism} for clarity.
\end{defn}

\begin{rem}
The homomorphisms $\pi$ and $\mu$ are not morphisms in the classical category $\cs$; this will cause no trouble, but does present a danger of confusion.
\end{rem}

\begin{rem}
Thus, in our new definition of degenerate covariant homomorphism, we include all the usual nondegenerate covariant homomorphisms, and we add more, allowing the homomorphism $\pi$ of $A$ (but not the homomorphism $\mu$ of $C_0(G)$) to be degenerate.
\end{rem}

\begin{rem}
We wrote $M(B\otimes C^*(G))$, rather than the relative multiplier algebra $\wilde M(B\otimes C^*(G))$,
in the above diagram, because $\bar{\pi\otimes\id}$ will in general not map $\wilde M(A\otimes C^*(G))$ into $\wilde M(B\otimes C^*(G))$
since $\pi$ does not map $A$ into $B$.
\end{rem}

Although we have apparently enlarged the supply of covariant homomorphisms, in some sense we have not.
In \lemref{nd} below we use the following terminology:
given $C^*$-algebras $A\subset B$, the \emph{idealizer} of $A$ in $B$ is $\{b\in B:bA\cup Ab\subset A\}$.

\begin{lem}\label{nd}
Let $(\pi,\mu)$ be a degenerate covariant homomorphism of $(A,\delta)$ to $B$, as in \defnref{cov def}.
Put
\[
B_0=\clspn\{\pi(A)\mu(C_0(G))\}.
\]
Then:
\begin{enumerate}
\item
$B_0=\clspn\{\mu(C_0(G))\pi(A)\}$.

\item
$B_0$ is a $C^*$-subalgebra of $M(B)$.

\item
$\pi$ and $\mu$ map into the idealizer $D$ of $B_0$ in $M(B)$.
Let $\rho:D\to M(B_0)$
be the homomorphism
given by
\[
\rho(m)b_0=mb_0\righttext{for}m\in D\subset M(B),b_0\in B_0\subset B,
\]
and let $\pi_0=\rho\circ\pi:A\to M(B_0)$ and $\mu_0=\rho\circ\mu:C_0(G)\to M(B_0)$.
Then $(\pi_0,\mu_0)$ is a nondegenerate covariant homomorphism of $(A,\delta)$ to $B_0$.

\item
For all $a\in A$ and $f\in C_0(G)$ we have
\[
\pi_0(a)\mu_0(f)=\pi(a)\mu(f)\in B_0.
\]
\end{enumerate}
\end{lem}

\begin{proof}
For (1), by symmetry
it suffices to show that for $a\in A$ and $f\in C_0(G)$ we have
\[
\mu(f)\pi(a)\in B_0,
\]
and we use an old trick from \cite[proof of Lemma~2.5]{lprs}:
since $A(G)$ is dense in $C_0(G)$,
it suffices to take $f\in A(G)$,
and then since $A(G)$ is a nondegenerate $C^*(G)$-module via $\<y,g\cdot x\>=\<xy,g\>$ for $x,y\in C^*(G),g\in A(G)$,
by Cohen's Factorization Theorem we can write $f=g\cdot x$.
Then the following approximation suffices:
\begin{align*}
\mu(f)\pi(a)
&=\<(\mu\otimes\id)(w_G),\id\otimes f\>\pi(a)
\\&=\<(\mu\otimes\id)(w_G)(\pi(a)\otimes 1),\id\otimes f\>
\\&=\<\bar{\pi\otimes\id}(\delta(a))(\mu\otimes\id)(w_G),\id\otimes g\cdot x\>
\\&=\<(\pi\otimes\id)((1\otimes x)\delta(a))(\mu\otimes\id)(w_G),\id\otimes g\>
\\&\approx \sum_i\<(\pi\otimes\id)(a_i\otimes x_i)(\mu\otimes\id)(w_G),\id\otimes g\>
\\&\hspace{1in}\text{for finitely many $a_i\in A,x_i\in C^*(G)$}
\\&= \sum_i\<(\pi(a_i)\otimes x_i)(\mu\otimes\id)(w_G),\id\otimes g\>
\\&= \sum_i\pi(a_i)\<(\mu\otimes\id)(w_G),\id\otimes g\cdot x_i\>
\\&= \sum_i\pi(a_i)\mu(g\cdot x_i).
\end{align*}
From (1) it follows that $B_0$ is a $*$-subalgebra of $B$, giving (2).

(3).
It is now clear that
\[
\pi(A)B_0\cup B_0\pi(A)\subset B_0,
\]
and similarly for $\mu$,
so both $\pi$ and $\mu$ map into $D$.
It is also clear that $\pi_0$ and $\mu_0$ map nondegenerately into $M(B_0)$.
The covariance property for $(\pi_0,\mu_0)$ follows quickly from that of $(\pi,\mu)$:
if $a\in A$ then
\begin{align*}
\ad (\mu_0\otimes\id)(w_G)(\pi_0(a)\otimes 1)
&=(\rho\otimes\id)\circ \ad (\mu\otimes\id)(w_G)(\pi(a)\otimes 1)
\\&=(\rho\otimes\id)\circ \bar{\pi\otimes\id}\circ\delta(a)
\\&=\bar{\pi_0\otimes\id}\circ\delta(a).
\end{align*}

(4).
This follows from the construction.
\end{proof}

Let $(A\rtimes_\delta G,j_A,j_G)$ be the usual crossed product of the coaction $(A,\delta)$,
i.e., $(j_A,j_G)$ is a nondegenerate covariant homomorphism of $(A,\delta)$ to $A\rtimes_\delta G$
that is universal in the sense that if $(\pi,\mu)$ is any nondegenerate covariant homomorphism of $(A,\delta)$ to a $C^*$-algebra $B$,
then there is a unique homomorphism
$\pi\times\mu:A\rtimes_\delta G\to M(B)$
such that
\begin{align*}
\pi\times\mu\circ j_A&=\pi\\
\pi\times\mu\circ j_G&=\mu,
\end{align*}
equivalently such that
\begin{equation}\label{universal}
\pi\times\mu\bigl(j_A(a)j_G(f)\bigr)=\pi(a)\mu(f)\righttext{for all}a\in A,f\in C_0(G).
\end{equation}

\begin{cor}\label{same coaction product}
With the above notation, $(j_A,j_G)$ is also universal among degenerate covariant homomorphisms \(in the sense of \defnref{cov def}\).
More precisely:
for any degenerate covariant homomorphism $(\pi,\mu)$ of $(A,\delta)$ to $B$ as in \defnref{cov def},
there is a unique homomorphism $\pi\times\mu:A\rtimes_\delta G\to M(B)$
satisfying \eqref{universal}.
\end{cor}

\begin{proof}
Let $\pi_0,\mu_0,B_0$ be as in the preceding lemma.
Then we have a unique homomorphism $\pi_0\times\mu_0:A\rtimes_\delta G\to M(B_0)$
such that
\[
\pi_0\times\mu_0\bigl(j_A(a)j_G(f)\bigr)=\pi_0(a)\mu_0(f)\righttext{for all}a\in A,f\in C_0(G).
\]
By construction we have $\pi\times\mu(A\rtimes_\delta G)\subset B_0$.
Since $B_0\subset M(B)$,
we can regard $\pi_0$ as a homomorphism $\pi:A\to M(B)$,
and similarly for $\mu:C_0(G)\to M(B)$.
Then we regard $\pi_0\times\mu_0$ as a homomorphism $\pi\times\mu:A\rtimes_\delta G\to M(B)$,
and trivially \eqref{universal} holds.
Since $\pi_0(a)\mu_0(f)=\pi(a)\mu(f)\in B_0$ for all $a\in A,f\in C_0(G)$,
the homomorphism $\pi\times\mu$ is unique.
\end{proof}

Similarly, and more easily, for actions:

\begin{defn}
In the \emph{classical category $\ac$ of actions}, a morphism $\phi:(A,\alpha)\to (B,\beta)$ is a morphism $\phi:A\to B$ in $\cs$ such that
\[
\beta_s\circ\phi=\phi\circ \alpha_s\righttext{for all}s\in G.
\]
\end{defn}

\begin{defn}\label{cov act}
A \emph{degenerate covariant homomorphism} of an action $(A,\alpha)$ to a $C^*$-algebra is a pair $(\pi,u)$, where $\pi:A\to M(B)$ is a homomorphism and $u:G\to M(B)$ is a strictly continuous unitary homomorphism such that
\[
\pi\circ\alpha_s=\ad u_s\circ\pi\righttext{for all}s\in G.
\]
\end{defn}
We call $(\pi,u)$ \emph{nondegenerate} if $\pi:A\to M(B)$ is.

\begin{lem}\label{nd act}
Let $(\pi,u)$ be a degenerate covariant homomorphism of an action $(A,\alpha)$ to $B$,
and put
\[
B_0=\clspn\{\pi(A)u(C^*(G))\},
\]
where we use the same notation $u$ for the associated nondegenerate homomorphism $u:C^*(G)\to M(B)$.
Then:
\begin{enumerate}
\item
$B_0=\clspn\{u(C^*(G))\pi(A)\}$.

\item
$B_0$ is a $C^*$-subalgebra of $M(B)$.

\item
$\pi$ and $u$ map into the idealizer $D$ of $B_0$ in $M(B)$.
Let $\rho:D\to M(B_0)$
be the homomorphism
given by
\[
\rho(m)b_0=mb_0\righttext{for}m\in D\subset M(B),b_0\in B_0\subset B,
\]
and let $\pi_0=\rho\circ\pi:A\to M(B_0)$ and $u_0=\rho\circ u:G\to M(B_0)$.
Then $(\pi_0,u_0)$ is a nondegenerate covariant homomorphism of $(A,\alpha)$ to $B_0$.

\item
For all $a\in A$ and $c\in C^*(G)$ we have
\[
\pi_0(a)u_0(c)=\pi(a)u(c)\in B_0.
\]
\end{enumerate}
\end{lem}

Let $(A\rtimes_\alpha G,i_A,i_G)$ be the usual crossed product of the action $(A,\alpha)$,
i.e., $(i_A,i_G)$ is a nondegenerate covariant homomorphism of $(A,\alpha)$ to $A\rtimes_\alpha G$
that is universal in the sense that if $(\pi,u)$ is any nondegenerate covariant homomorphism of $(A,\alpha)$ to a $C^*$-algebra $B$,
then there is a unique homomorphism
$\pi\times u:A\rtimes_\alpha G\to M(B)$
such that
\begin{equation}\label{universal act}
\pi\times u\bigl(i_A(a)i_G(c)\bigr)=\pi(a)u(c)\righttext{for all}a\in A,c\in C^*(G).
\end{equation}

\begin{cor}\label{same product}
With the above notation, $(i_A,i_G)$ is also universal among degenerate covariant homomorphisms \(in the sense of \defnref{cov def}\):
for any degenerate covariant homomorphism $(\pi,u)$ of $(A,\alpha)$ to $B$ as in \defnref{cov act},
there is a unique homomorphism $\pi\times u:A\rtimes_\alpha G\to M(B)$
satisfying \eqref{universal act}.
\end{cor}

If $\phi:(A,\delta)\to (B,\epsilon)$ in $\co$,
then a routine adaptation of the usual arguments shows that we get a morphism
\[
\phi\rtimes G=(j_B\circ\phi)\times j_G^B:(A\rtimes_\delta G,\what\delta)\to (B\rtimes_\epsilon G,\what\epsilon)
\]
in $\ac$,
and similarly if $\phi:(A,\alpha)\to (B,\beta)$ in $\ac$ we get a morphism
\[
\phi\rtimes G=(i_B\circ\phi)\times i_G^B:(A\rtimes_\alpha G,\what\alpha)\to (B\rtimes_\beta G,\what\beta)
\]
in $\co$.
Thus we have crossed-product functors between the classical categories of coactions and actions.

It is also routine to verify that 
if $(A,\delta)$ is a coaction then
the canonical surjection
\[
\Phi:A\rtimes_\delta G\rtimes_{\what\delta} G\to A\otimes\KK
\]
is a natural transformation between the double crossed-product functor and stabilization.\footnote{It is completely routine to verify that stabilization $A\mapsto A\otimes\KK$ is a functor on the classical category $\cs$.}

We need to check that normalization and maximalization behave appropriately in the new coaction category.

\subsection*{Maximalization}

A \emph{maximalization} of a coaction $(A,\delta)$ consists of a maximal coaction $(A^m,\delta^m)$ and a surjective morphism $q^m:(A^m,\delta^m)\to (A,\delta)$ in $\co$ such that
\[
q^m\rtimes G:A^m\rtimes_{\delta^m} G\to A\rtimes_\delta G
\]
is an isomorphism.
Existence of maximalizations is established in \cite[Theorem~6.4]{fischer}, \cite[Theorem~3.3]{maximal}.

To make maximalization into a functor on the classical category of coactions, we 
note that the argument of \cite[Proof of Lemma~6.2]{fischer} carries over to give an appropriate version of the universal property:
given coactions $(A,\delta)$ and $(B,\epsilon)$,
with $\epsilon$ maximal,
and a morphism $\phi:(B,\epsilon)\to (A,\delta)$ in $\co$,
there is a unique morphism $\wilde\phi$ in $\co$ making the diagram
\[
\xymatrix@C+10pt{
(B,\epsilon) \ar@{-->}[r]^-{\wilde\phi} \ar[dr]_\phi
&(A^m,\delta^m) \ar[d]^{q^m}
\\
&(A,\delta)
}
\]
commute.
Thus, given a morphism $\phi:(A,\delta)\to (B,\epsilon)$ in $\co$,
there is a unique morphism $\phi^m$ making the diagram
\[
\xymatrix@C+10pt{
(A^m,\delta^m) \ar@{-->}[r]^-{\phi^m} \ar[d]_{q^m_A}
&(B^m,\epsilon^m) \ar[d]^{q^m_B}
\\
(A,\delta) \ar[r]_-\phi
&(B,\epsilon)
}
\]
commute in $\co$. Uniqueness makes the assignments $\phi\mapsto \phi^m$ functorial,
and the \emph{maximalizing maps} $q^m$ give a natural transformation from the maximalization functor to the identity functor.
Also, the universal property implies that the maximalization functor is faithful,
i.e., if $\phi,\psi:(A,\delta)\to (B,\epsilon)$ are distinct morphisms in $\co$,
then the maximalizations $\phi^m,\psi^m:(A^m,\delta^m)\to (B^m,\epsilon^m)$ are also distinct.

\begin{rem}\label{choice}
It is important for us that maximalization is a \emph{functor};
however, when we refer to $(A^m,\delta^m)$ as ``the'' maximalization of a coaction $(A,\delta)$,
we do not have in mind a specific $C^*$-algebra $A^m$, rather we regard the maximalization as being characterized up to isomorphism by its universal properties, but for the purpose of having a functor we imagine that a choice of maximalization has been made for every coaction --- any other choices would give a naturally isomorphic functor.
On the other hand, whenever we have a maximal coaction $(B,\epsilon)$, we may call a morphism $\phi:(B,\epsilon)\to (A,\delta)$ with the defining property \emph{a maximalization} of $(A,\delta)$.
\end{rem}

\subsection*{Normalization}

A \emph{normalization} of a coaction $(A,\delta)$ consists of a normal coaction $(A^n,\delta^n)$ and a surjective morphism $\Lambda:(A,\delta)\to (A^n,\delta^n)$ in $\co$ such that
\[
\Lambda\rtimes G:A\rtimes_{\delta} G\to A^n\rtimes_{\delta^n} G
\]
is an isomorphism.
Existence of normalizations is established in \cite[Proposition~2.6]{fullred}.

To make normalization into a functor on the classical category of coactions, we 
note that
\cite[Lemma~2.1]{maximal} says that,
given a morphism $\phi:(A,\delta)\to (B,\epsilon)$ in $\co$,
there is a unique morphism $\phi^n$ making the diagram
\[
\xymatrix@C+10pt{
(A,\delta) \ar[r]^-{\phi} \ar[d]_{\Lambda_A}
&(B,\epsilon) \ar[d]^{\Lambda_B}
\\
(A^n,\delta^n) \ar@{-->}[r]_-{\phi^n}
&(B^n,\epsilon^n)
}
\]
commute in $\co$. Uniqueness makes the assignments $\phi\mapsto \phi^n$ functorial,
and the \emph{normalizing maps} $\Lambda$ give a natural transformation from the identity functor to the normalization functor.

\begin{rem}
The comments of \remref{choice} can be adapted in an obvious way to normalization,
and also to crossed products, etc.
There are numerous ``natural'' relationships among such functors; for example, maximalization is naturally isomorphic to the composition
\[
(A,\delta)\mapsto (A^n,\delta) \mapsto (A^{nm},\delta^{nm})
\]
of normalization followed by maximalization,
and
the dual coaction $\what\alpha^n$ on the reduced crossed product $A\rtimes_{\alpha,r} G$ of an action $(A,\alpha)$ is naturally isomorphic to the normalization
of the dual coaction $\what\alpha$ on the full crossed product $A\rtimes_\alpha G$
\cite[Proposition~A.61]{enchilada}.
\end{rem}

The normalization $\Lambda:(A,\delta)\to (A^n,\delta^n)$ of a maximal coaction is also a maximalization of the normal coaction $\delta^n$.
It follows that the normalization functor is faithful,
i.e., if $\phi,\psi:(A,\delta)\to (B,\epsilon)$ are distinct morphisms in $\co$,
then the normalizations $\phi^n,\psi^n:(A^n,\delta^n)\to (B^n,\epsilon^n)$ are also distinct.
It follows from this
and surjectivity of
the normalizing maps $\Lambda_A:(A,\delta)\to (A^n,\delta^n)$
that the normalizing maps
are monomorphisms in the category $\co$,
i.e., if $\phi,\psi:(A,\delta)\to (B,\epsilon)$ are distinct morphisms in $\co$,
then the compositions $\Lambda_B\circ \phi,\Lambda_B\circ \psi:(A,\delta)\to (B^n,\epsilon^n)$ are also distinct.\footnote{The analogous fact for the nondegenerate category of coactions is
\cite[Corollary~6.1.20]{bkq1}.}

\subsection*{Exact sequences}

It is crucial for us to note that
in each of the classical categories $\cs$, $\co$, and $\ac$
there is an obvious concept of short exact sequence.
Nilsen \cite{nil:full} develops the basic theory of short exact sequences for coactions and crossed products.
We briefly outline the essential facts here.

\begin{defn}
Let $(A,\delta)$ be a coaction.
An ideal $I$ of $A$ is \emph{strongly $\delta$-invariant}
if
\[
\clspn\{\delta(I)(1\otimes C^*(G))\}=I\otimes C^*(G).
\]
We will normally just write \emph{invariant} to mean strongly invariant.
\end{defn}

Nilsen proves in \cite[Proposition~2.1, Proposition~2.2, Theorem~2.3]{nil:full}
(see also \cite[Proposition~4.8]{lprs})
that, with the conventions of \cite{nil:full}, if $I$ is strongly invariant then:
\begin{enumerate}
\item $\delta$ restricts to a coaction $\delta_I$ on $I$.

\item $I\rtimes_{\delta_I} G$ is (canonically isomorphic to) an ideal of $A\rtimes_\delta G$.

\item $I$ is \emph{weakly} $\delta$-invariant, i.e.,
$\delta$ descends to a coaction $\delta^I$ on $A/I$.

\item
$0\to I\rtimes_{\delta_I} G\to A\rtimes_\delta G\to (A/I)\rtimes_{\delta^I} G\to 0$
is a short exact sequence in the classical category $\cs$.
\end{enumerate}

We 
point out that Nilsen had to do a bit of work to map $I\rtimes_{\delta_I} G$ into $A\rtimes_\delta G$;
in our framework with the classical categories,
we just
note that the inclusion $\phi:I\hookrightarrow A$ is $\delta_I-\delta$ equivariant,
hence gives a morphism in $\co$,
so we can apply the functor $\cp$ to get a morphism
\[
\phi\rtimes G:I\rtimes_{\delta_I} G\to A\rtimes_\delta G\midtext{in}\cs.
\]

\begin{defn}\label{def:exact}
A functor between any two of the categories $\cs, \co, \ac$ is \emph{exact} if it preserves short exact sequences.
\end{defn}

\begin{ex}
The full crossed-product functor
\begin{align*}
(A,\alpha)&\mapsto (A\rtimes_\alpha G,\what\alpha)
\\\phi&\mapsto \phi\rtimes G
\end{align*}
from $\ac$ to $\co$ is exact \cite[Proposition~12]{gre:local}.
However, the reduced crossed-product functor is not exact, due to Gromov's examples of non-exact groups.
\end{ex}

\begin{ex}
The crossed-product functor 
\begin{align*}
(A,\delta)&\mapsto (A\rtimes_\delta G,\what\delta)
\\\phi&\mapsto \phi\rtimes G
\end{align*}
from $\co$ to $\ac$ is exact \cite[Theorem~2.3]{nil:full}.
\end{ex}

\begin{ex}
The stabilization functor
\begin{align*}
A&\mapsto A\otimes\KK
\\\phi&\mapsto \phi\otimes\id
\end{align*}
on $\cs$ is exact.
\end{ex}

\section{Coaction functors}\label{sec:coaction functor}

\cite{bgwexact} defined a \emph{crossed-product} as a functor
$(B,\alpha)\mapsto B\rtimes_{\alpha,\tau} G$,
from the category of actions to the category of $C^*$-algebras,
equipped with natural transformations
\[
\xymatrix{
B\rtimes_\alpha G \ar[r] \ar[d]
&B\rtimes_{\alpha,\tau} G \ar[dl]
\\
B\rtimes_{\alpha,r} G,
}
\]
where the vertical arrow is the regular representation,
such that the horizontal arrow is surjective.

Our predilection is to decompose such a crossed-product functor as a composition
\[
(B,\alpha)\mapsto (B\rtimes_\alpha G,\what\alpha) \mapsto B\rtimes_{\alpha,\tau} G,
\]
where the first arrow is the full crossed product and the second arrow depends only upon the dual coaction $\what\alpha$.
Our approach will require the target $C^*$-algebra $B\rtimes_{\alpha,\tau} G$ to carry a quotient of the dual coaction.
Thus, it is certainly not obvious that our techniques can handle all crossed-product functors of \cite{bgwexact},
because
\cite{bgwexact} do not require their crossed products $B\rtimes_{\alpha,\tau} G$ to have coactions, and
even if they all do, there is no reason to believe that the crossed-product functor factors in this way.
Nevertheless, we think that it is useful to study crossed-product functors that do factor, and thus we can focus upon the second functor, where all the action stays within the realm of coactions.
The following definition is adapted more or less directly from \cite[Definition~2.1]{bgwexact}:

\begin{defn}\label{coaction functor}
A \emph{coaction functor} is a functor 
$\tau:(A,\delta)\mapsto (A^\tau,\delta^\tau)$
on
the category of coactions,
together with
a natural transformation
$q^\tau$ from maximalization to $\tau$
such that for every coaction $(A,\delta)$,
\begin{enumerate}
\item
$q^\tau_A:A^m\to A^\tau$ is surjective, and

\item
$\ker q^\tau_A\subset \ker \Lambda_{A^m}$.
\end{enumerate}
\end{defn}

\begin{ex}
\begin{enumerate}
\item
Maximalization $(A,\delta)\mapsto (A^m,\delta^m)$ is a coaction functor, with natural surjections 
given by the identity maps $\id_{A^m}$.

\item
Normalization $(A,\delta)\mapsto (A^n,\delta^n)$ is a coaction functor, with natural surjections $\Lambda_{A^m}:A^m\to A^n$.

\item
The identity functor is a coaction functor, with  with natural surjections
$q^m_A:A^m\to A$.
\end{enumerate}
\end{ex}

\begin{lem}\label{axiom}
If $\tau$ is a coaction functor, then for every coaction $(A,\delta)$
there is a unique
$\delta^\tau-\delta^n$ equivariant 
surjection $\Lambda^\tau_A$ making the diagram
\begin{equation}\label{Lambda tau}
\xymatrix{
A^m \ar[r]^-{q^\tau_A} \ar[d]_{\Lambda_{A^m}}
&A^\tau \ar@{-->}[dl]^{\Lambda^\tau_A}
\\
A^n
}
\end{equation}
commute.
Moreover, $\Lambda^\tau$ is a natural transformation from $\tau$ to normalization.
\end{lem}

\begin{proof}
The first statement follows immediately from the definitions.
To verify that $\Lambda^\tau$ is a natural transformation,
we must show that
the homomorphisms $\Lambda^\tau$ 
\begin{enumerate}
\item
are morphisms of coactions, and

\item
are natural.
\end{enumerate}

(1)
In the commuting triangle \eqref{Lambda tau},
we must show that $\Lambda^\tau_A$ is a $B(G)$-module map,
but this follows since $\Lambda_{A^m}$ and $q^\tau_A$ are module maps
and $q^\tau_A$ is surjective.

(2)
For the naturality,
let $\phi:(A,\delta)\to (B,\epsilon)$ be a morphism in the category of coactions.
Consider the diagram
\[
\xymatrix{
A^m \ar[rr]^{\phi^m} \ar[dd]_{\Lambda_{A^m}} \ar[dr]^{q^\tau_A}
&&B^m \ar[dr]^{q^\tau_B} \ar'[d][dd]_(.4){\Lambda_{B^m}}
\\
&A^\tau \ar[rr]^(.3){\phi^\tau} \ar[dl]^(.4){\Lambda^\tau_A}
&&B^\tau \ar[dl]^{\Lambda^\tau_B}
\\
A^n \ar[rr]_-{\phi^n}
&&B^n.
}
\]
We need to know that the lower quadrilateral, with horizontal and southwest arrows, commutes,
and this follows from
surjectivity of $q^\tau_A$ and
commutativity of the other two quadrilaterals and the two triangles.
\end{proof}

\begin{cor}
If $\tau$ is a coaction functor, then in \eqref{Lambda tau} we have:
\begin{enumerate}
\item
$q^\tau:A^m\to A^\tau$ is a maximalization of $\delta^\tau$, and

\item
$\Lambda^\tau:A^\tau\to A^n$ is a normalization of $\delta^\tau$.
\end{enumerate}
\end{cor}

\begin{proof}
Taking crossed products in \eqref{Lambda tau},
we get a commutative diagram
\[
\xymatrix@C+30pt@R+20pt{
A^m\rtimes_{\delta^m} G \ar[r]^-{q^\tau\rtimes G}_-\simeq \ar[d]_{\Lambda\rtimes G}^\simeq
&A^\tau\rtimes_{\delta^\tau} G \ar[dl]^{\Lambda^\tau\rtimes G}_-\simeq
\\
A^n\rtimes_{\delta^n} G,
}
\]
where the horizontal arrow is surjective because $q^\tau$ is, and is injective because of the vertical isomorphism, and then the diagonal arrow is an isomorphism because the other two arrows are.
Thus $q^\tau$ and $\Lambda^\tau$ satisfy the defining properties of maximalization and normalization, respectively.
\end{proof}

\begin{rem}
Caution: it might seem that $\tau$ should factor through the maximalization functor, at least up to natural isomorphism.
This would entail, in particular, that
\[
(A^{m\tau},\delta^{m\tau})\cong (A^\tau,\delta^\tau)\righttext{for every coaction $(A,\delta)$}.
\]
But this is violated with $\tau=\id$.
\end{rem}

\begin{notn}
With the above notation, we define an ideal of $A^m$ by
\[
A^m_\tau:=\ker q^\tau_A.
\]
Note that for the maximalization functor $m$  we have $A^m_m=\{0\}$,
while for the normalization functor $n$ the associated ideal $A^m_n$ is the kernel of the normalization map $\Lambda_{A^m}:A^m\to A^{mn}\cong A^n$.
\end{notn}

\subsection*{Partial ordering of coaction functors}
\cite[see page~8]{bgwexact} defines one crossed-product functor $\sigma$ to be \emph{smaller} than another one $\tau$ if the natural surjection $A\rtimes_{\alpha,\tau} G\to A\rtimes_{\alpha,r} G$ factors through the $\sigma$-crossed product.

We adapt this definition of partial order to coaction functors,
but ``from the top rather than toward the bottom''.

\begin{defn}\label{smaller}
If $\sigma$ and $\tau$ are coaction functors, then
$\sigma$ is \emph{smaller} than $\tau$, written $\sigma\le \tau$, if 
for every coaction $(A,\delta)$ we have
\[
A^m_\tau\subset A^m_\sigma.
\]
\end{defn}

\begin{lem}\label{order}
For coaction functors $\sigma,\tau$, the following are equivalent:
\begin{enumerate}
\item $\sigma\le \tau$.

\item For every coaction $(A,\delta)$ there is a homomorphism $\Gamma^{\tau,\sigma}$ making the diagram
\[
\xymatrix{
A^m \ar[r]^-{q^\tau} \ar[dr]_{q^\sigma}
&A^\tau \ar@{-->}[d]^{\Gamma^{\tau,\sigma}}
\\
&A^\sigma
}
\]
commute.

\item For every coaction $(A,\delta)$ there is a homomorphism $\Gamma^{\tau,\sigma}$ making the diagram
\[
\xymatrix{
&A^\tau \ar[dl]_{\Lambda^\tau} \ar@{-->}[d]^{\Gamma^{\tau,\sigma}}
\\
A^n
&A^\sigma \ar[l]^{\Lambda^\sigma}
}
\]
commute.
\end{enumerate}
Moreover, if the above equivalent conditions hold then $\Gamma^{\tau,\sigma}$ is
unique, is
surjective, and is a natural transformation from $\tau$ to $\sigma$.
\end{lem}

\begin{proof}
(1) is equivalent to (2) since $A^m_\tau=\ker q^\tau$ and $A^m_\sigma=\ker q^\sigma$.
Moreover, (1) implies that $\Gamma^{\tau,\sigma}$ is unique
and is surjective,
since the maps $q^\tau$ are surjective.

Assume (3).
Consider the combined diagram
\begin{equation}
\begin{split}\label{combined}
\xymatrix@+30pt{
A^m \ar[r]^-{q^t} \ar[dr] ^(.3){q^\sigma} |!{[r];[d]}\hole \ar[d]_{\Lambda_{A^m}}
&A^\tau \ar[dl]_(.3){\Lambda^\tau} \ar[d]^{\Gamma^{\tau,\sigma}}
\\
A^n
&A^\sigma. \ar[l]^{\Lambda^\sigma}
}
\end{split}
\end{equation}
The upper left and lower left triangles commute by definition of coaction functor,
and the lower right triangle commutes by assumption.
Thus the upper right triangle commutes after post-composing with $\Lambda^\sigma$.
Since the latter map is a normalizer,
by \cite[Corollary~6.1.20]{bkq1}
it is a monomorphism in the category of coactions.
Thus the upper right triangle commutes.

Similarly (but more easily), assuming (2),
the lower right triangle in the diagram \eqref{combined} commutes
because it commutes
after pre-composing with the surjection $q^\tau$.

Naturality of $\Gamma^{\tau,\sigma}$ can be proved by essentially the same argument as in \lemref{axiom}.
\end{proof}

The following is a coaction-functor analogue of \cite[Lemma~3.7]{bgwexact}, and we adapt their argument:
\begin{thm}\label{glb}
Every nonempty collection $\TT$ of coaction functors has a greatest lower bound $\sigma$ with respect to the above partial ordering,
characterized by
\[
A^m_\sigma=\clspn_{\tau\in\TT}A^m_\tau
\]
for every coaction $(A,\delta)$.
\end{thm}

\begin{proof}
Let $(A,\delta)$ be a coaction,
Then the ideal
\[
A^m_\sigma:=\clspn_{\tau\in\TT}A^m_\tau
\]
of $A^m$ is contained in the kernel of the normalization map $\Lambda_{A^m}$.
Put
\[
A^\sigma=A^m/A^m_\sigma,
\]
and let
\[
q^\sigma_A:A^m\to A^\sigma
\]
be the quotient map.

For all $\tau\in \TT$,
$A^m_\tau$ is a weakly $\delta^m$-invariant ideal of $A^m$,
so for all $f\in B(G)$ we have
\[
f\cdot A^m_\tau\subset A^m_\tau\subset A^m_\sigma,
\]
and
it follows that $f\cdot A^m_\sigma\subset A^m_\sigma$,
i.e.,
$A^m_\sigma$ is a weakly $\delta^m$-invariant ideal.
Thus
$q^\sigma$ is equivariant for $\delta^m$ and 
a unique coaction $\delta^\sigma$ on $A^\sigma$.

We now have assignments
\[
(A,\delta)\mapsto (A^\sigma,\delta^\sigma)
\]
on objects,
and we need to handle morphisms.
Thus, let $\phi:(A,\delta)\to (B,\epsilon)$ be a morphism of coactions,
i.e., $\phi:A\to B$ is a $\delta-\epsilon$ equivariant homomorphism.
Since
\[
A^m_\tau\subset (\phi^m)\inv(B^m_\tau)\subset (\phi^m)\inv(B^m_\sigma)
\midtext{for all $\tau\in\TT$,}
\]
we have
\[
\ker q^\sigma_A=
A^m_\sigma=\clspn_{\tau\in\TT}A^m_\tau\subset (\phi^m)\inv(B^m_\sigma)
=\ker q^\sigma_B\circ \phi^m.
\]
Thus there is a unique homomorphism $\phi^\sigma$
making the diagram
\begin{equation}\label{phi tau}
\xymatrix{
A^m \ar[r]^-{\phi^m} \ar[d]_{q^\sigma_A}
&B^m \ar[d]^{q^\sigma_B}
\\
A^\sigma \ar@{-->}[r]_-{\phi^\sigma}
&B^\sigma
}
\end{equation}
commute.
Moreover,
$\phi^\sigma$ is $\delta^\sigma-\epsilon^\sigma$ equivariant
because the other three maps are and $q^\sigma_A$ is surjective.

We need to verify that the assignments $\phi\mapsto \phi^\sigma$ of morphisms are functorial.
Obviously identity morphisms are preserved.
For compositions,
let
\[
\xymatrix{
(A,\delta) \ar[r]^-\phi \ar[dr]_\nu
&(B,\epsilon) \ar[d]^\rho
\\
&(C,\gamma)
}
\]
be a commuting diagram of coactions.
Consider the diagram
\[
\xymatrix{
A^m \ar[rr]^{\phi^m} \ar[dd]_{q^\sigma_A} \ar[dr]_{\nu^m}
&&B^m \ar[dd]^{q^\sigma_B} \ar[dl]^{\rho^m} 
\\
&C^m \ar[dd]^(.3){q^\sigma_C}
\\
A^\tau \ar[rr]|!{[ur];[dr]}\hole^(.3){\phi^\tau} \ar[dr]_{\nu^\tau}
&&B^\tau \ar[dl]^{\rho^\tau}
\\
&C^\tau.
}
\]
The three 
vertical
quadrilaterals and the 
top
triangle commute,
and $q^\sigma_A$ is surjective.
It follows that the 
bottom
triangle commutes,
and we have shown that composition is preserved.

Thus we have a functor $\sigma$ on the category of coactions.
Moreover,
$\sigma$
is a coaction functor,
since
the surjections $q^\sigma$
have small kernels and
the commuting diagram \eqref{phi tau} shows that 
$q^\sigma$ gives
a natural transformation from maximalization to $\sigma$.
By construction, $\sigma$ is a greatest lower bound for $\TT$.
\end{proof}

\subsection*{Exact coaction functors}

As a special case of our general \defnref{def:exact}, we explicitly record:

\begin{defn}\label{exact functor}
A coaction functor $\tau$ is \emph{exact} if for every short exact sequence
\[
\xymatrix{
0 \ar[r] 
&(I,\gamma) \ar[r]^-\phi
&(A,\delta) \ar[r]^-\psi
&(B,\epsilon) \ar[r] &0
}
\]
of coactions the associated sequence
\[
\xymatrix{
0 \ar[r] 
&(I^\tau,\gamma^\tau) \ar[r]^-{\phi^\tau}
&(A^\tau,\delta^\tau) \ar[r]^-{\psi^\tau}
&(B^\tau,\epsilon^\tau) \ar[r] &0
}
\]
is exact.
\end{defn}

\begin{thm}\label{max exact}
The maximalization functor is exact.
\end{thm}

\begin{proof}
Let 
\[
\xymatrix{
0 \ar[r] 
&(I,\gamma) \ar[r]^-\phi
&(A,\delta) \ar[r]^-\psi
&(B,\epsilon) \ar[r] &0
}
\]
be an exact sequence of coactions.
Taking crossed products twice, we get an exact sequence
\[
\xymatrix@C+20pt{
0 \ar[r]
&I\rtimes_\gamma G\rtimes_{\what\gamma} G \ar[r]^{\phi\rtimes G\rtimes G}
&A\rtimes_\delta G\rtimes_{\what\delta} G
\\&{\hphantom{I\rtimes_\gamma G\rtimes_{\what\gamma} G}}
\ar[r]^{\psi\rtimes G\rtimes G}
&B\rtimes_\epsilon G\rtimes_{\what\epsilon} G
\ar[r]&0.
}
\]
Since the identity functor on coactions is a coaction functor, we get an isomorphic sequence
\[
\xymatrix@C+20pt{
0 \ar[r]
&I^m\rtimes_{\gamma^m} G\rtimes_{\what{\gamma^m}} G \ar[r]^{\phi^m\rtimes G\rtimes G}
&A^m\rtimes_{\delta^m} G\rtimes_{\what{\delta^m}} G
\\&{\hphantom{I\rtimes_{\gamma^m} G\rtimes_{\what\gamma^m} G}}
\ar[r]^{\psi^m\rtimes G\rtimes G}
&B^m\rtimes_{\epsilon^m} G\rtimes_{\what{\epsilon^m}} G
\ar[r]&0,
}
\]
which is therefore also exact.
Since the canonical surjection $\Phi$ is a natural transformation from the double crossed-product functor to the stabilization functor, and since the coactions are now maximal, we get an isomorphic sequence
\[
\xymatrix@C+10pt{
0\ar[r]
&I^m\otimes\KK \ar[r]^{\phi^m\otimes\id}
&A^m\otimes\KK \ar[r]^{\psi^m\otimes\id}
&B^m\otimes\KK
\ar[r]&0,
}
\]
which is therefore also exact.
Since $\KK$ is an exact $C^*$-algebra,
\[
(\ker \phi^m)\otimes\KK=\ker (\phi^m\otimes\id)=\{0\},
\]
so $\ker \phi^m=\{0\}$,
and similarly
\[
(\ker \psi^m)\otimes\KK=\ker (\psi^m\otimes\id)=(\phi^m\otimes\id)(I^m\otimes\KK)
=\phi^m(I^m)\otimes\KK,
\]
so, because $\phi^m(I^m)\subset \ker \psi^m$ by functoriality,
we must have $\phi^m(I^m)=\ker \psi^m$.
Therefore
the sequence
\[
\xymatrix{
0\ar[r]
&I^m \ar[r]^{\phi^m}
&A^m \ar[r]^{\psi^m}
&B^m
\ar[r]&0
}
\]
is exact.
\end{proof}

\begin{thm}\label{functor exact}
A coaction functor $\tau$ is exact if and only if
for 
any short exact sequence
\[
\xymatrix{
0\ar[r]&(I,\delta_I)\ar[r]^\phi&(A,\delta)\ar[r]^\psi&(B,\delta^I)\ar[r]&0
}
\]
of coactions, both
\[
\phi^m(I^m_\tau)=\phi^m(I^m)\cap A^m_\tau
\]
and
\[
\phi^m(I^m)+A^m_\tau=(\psi^m)\inv(B^m_\tau)
\]
hold.
\end{thm}

\begin{proof}
We have a commutative diagram
\begin{equation}\label{exact diagram}
\xymatrix{
&0\ar[d]&0\ar[d]&0\ar[d]
\\
0\ar[r]&I^m_\tau\ar[d]_{\iota_I}\ar[r]^{\phi^m|}&A^m_\tau\ar[d]_{\iota_A}\ar[r]^{\psi^m|}&B^m_\tau\ar[d]_{\iota_{B}}\ar[r]&0
\\
0\ar[r]&I^m\ar[d]_{q_I}\ar[r]^{\phi^m}&A^m\ar[d]_{q_A}\ar[r]^{\psi^m}&B^m\ar[d]_{q_{B}}\ar[r]&0
\\
0\ar[r]&I^\tau\ar[r]^{\phi^\tau}\ar[d]&A^\tau\ar[r]^{\psi^\tau}\ar[d]&B^\tau\ar[r]\ar[d]&0
\\
&0&0&0,
}
\end{equation}
where the columns are exact by definition,
and the middle row is exact by \thmref{max exact}.
Thus the result follows immediately from \lemref{nine}.
\end{proof}

\subsection*{Morita compatible coaction functors}

If we have coactions $(A,\delta)$ and $(B,\epsilon)$,
and a $\delta-\epsilon$ compatible coaction $\zeta$ on an $A-B$ imprimitivity bimodule $X$,
we'll say that $(X,\zeta)$ is an \emph{$(A,\delta)-(B,\epsilon)$ imprimitivity bimodule}.

\begin{ex}\label{double dual}
The double dual bimodule coaction
\[
(Y,\eta):=\bigl(X\rtimes_\zeta G\rtimes_{\what\zeta} G,\what{\what\zeta}\,\bigr)
\]
is an
\[
\bigl(A\rtimes_\delta G\rtimes_{\what\delta} G,\what{\what\delta}\,\bigr)-
\bigl(B\rtimes_\epsilon G\rtimes_{\what\epsilon} G,\what{\what\epsilon}\,\bigr)
\]
imprimitivity bimodule.
Since the identity functor on coactions is a coaction functor,
$(Y,\eta)$ becomes an
\[
\bigl(A^m\rtimes_{\delta^m} G\rtimes_{\what{\delta^m}} G,\what{\what{\delta^m}}\,\bigr)-
\bigl(B^m\rtimes_{\epsilon^m} G\rtimes_{\what{\epsilon^m}} G,\what{\what{\epsilon^m}}\,\bigr)
\]
imprimitivity bimodule.
Since maximalizations satisfy full-crossed-product duality,
$(Y,\eta)$ becomes, after replacing the double dual coactions by exterior equivalent coactions,
an
\[
(A^m\otimes\KK,\delta^m\otimes_* \id)-(B^m\otimes\KK,\epsilon^m\otimes_* \id)
\]
imprimitivity bimodule
(see \cite[Lemma~3.6]{maximal}).
\end{ex}

We need the following basic lemma, which is probably folklore, 
although we could not find it in the literature.
Our formulation is partially inspired by Fischer's treatment of relative commutants of $\KK$
\cite[Section~3]{fischer}.

\begin{lem}\label{tensor}
Let $A$ and $B$ be $C^*$-algebras, and let $Y$ be an $(A\otimes\KK)-(B\otimes\KK)$ imprimitivity bimodule. Define
\[
X=\{m\in M(Y):(1_A\otimes k)\cdot m=m\cdot (1_B\otimes k)\in Y\text{ for all }k\in\KK\}.
\]
Then:
\begin{enumerate}
\item
$X$ is an $(A\otimes 1_\KK)-(B\otimes 1_\KK)$ submodule of $M(Y)$;

\item
$\clspn\<X,X\>_{M(B\otimes\KK)}=B\otimes 1_\KK$;

\item
$\clspn{}_{M(A\otimes\KK)}\<X,X\>=A\otimes 1_\KK$.
\end{enumerate}
Thus $X$ becomes an $A-B$ imprimitivity bimodule in an obvious way,
and moreover there is a unique
$(A\otimes\KK)-(B\otimes\KK)$ imprimitivity bimodule isomorphism
\[
\theta:X\otimes\KK\iso Y
\]
such that
\[
\theta(m\otimes k)=m\cdot (1_B\otimes k)\righttext{for}m\in X,k\in\KK.
\]
\end{lem}

\begin{lem}\label{Xm}
Given coactions $(A,\delta)$ and $(B,\epsilon)$,
and a $\delta-\epsilon$ compatible coaction $\zeta$ on an $A-B$ imprimitivity bimodule $X$,
let $(Y,\eta)$ be the
\[
(A^m\otimes\KK,\delta^m\otimes_* \id)-(B^m\otimes\KK,\epsilon^m\otimes_* \id)
\]
imprimitivity bimodule from \exref{double dual},
and let $X^m$ denote the associated $A^m-B^m$ imprimitivity bimodule as in \lemref{tensor},
with an $(A^m\otimes\KK)-(B^m\otimes\KK)$ imprimitivity bimodule isomorphism
$\theta:X^m\otimes\KK\to Y$.
Then there is a unique $\delta^m-\epsilon^m$ compatible coaction $\zeta^m$ on $X^m$ such that
$\theta$ transports $\zeta^m\otimes_*\id$ to $\eta$.
\end{lem}

\begin{proof}
The diagram
\[
\xymatrix@C+30pt{
X^m\otimes\KK 
\ar@{-->}[r]^-\kappa
\ar[d]_\theta^\simeq
&M(X^m\otimes\KK\otimes C^*(G)) \ar[d]^{\theta\otimes\id}_\simeq
\\
Y \ar[r]_-\eta
&M(Y\otimes C^*(G))
}
\]
certainly has a unique commuting completion, and $\kappa$ is a $(\delta^m\otimes_*\id)-(\epsilon^m\otimes_*\id)$ compatible coaction on $X^m\otimes\KK$.
In order to recognize that $\kappa$ is of the form $\zeta^m\otimes_*\id$,
we need to know that,
letting $\Sigma:\KK\otimes C^*(G)\to C^*(G)\otimes\KK$
be the flip isomorphism,
for every $\xi\in X^m$,
the element
\[
m:=(\id_{X^m}\otimes\Sigma)\circ(\theta\otimes\id)\inv\circ\eta\circ\theta(\xi\otimes 1_\KK)
\]
of the multiplier bimodule
$M(X^m\otimes C^*(G)\otimes \KK)$
is contained in the subset
$M(X^m\otimes C^*(G))\otimes 1_\KK$,
and for this we need only check that for all $k\in\KK$ we have
\[
(1_{A\otimes C^*(G)}\otimes k)\cdot m=m\cdot (1_{B\otimes C^*(G)}\otimes k)
\in X^m\otimes C^*(G)\otimes \KK,
\]
which follows from the properties of the maps involved.
Then it is routine to check that the resulting map $\zeta^m$ is a $\delta^m-\epsilon^m$ compatible coaction on $X^m$.
\end{proof}

\begin{defn}\label{morita functor}
A coaction functor $\tau$ is \emph{Morita compatible} if
whenever $(X,\zeta)$ is an
$(A,\delta)-(B,\epsilon)$ imprimitivity bimodule,
with associated $A^m-B^m$ imprimitivity bimodule $X^m$ as above,
the Rieffel correspondence of ideals satisfies
\begin{equation}\label{induce}
X^m\dashind B^m_\tau=A^m_\tau.
\end{equation}
\end{defn}
We will use without comment the simple observation that
if $(A,\delta)$ (and hence also $(B,\epsilon)$) is maximal,
then we can replace $X^m$ by $X$
and regard the natural surjection $q^\tau_A$ as going from $A$ to $A^\tau$
(and similarly for $B$),
since the maximalizing maps $q^m_A$ and $q^m_B$ can be combined to
give an isomorphism of the $A^m-B^m$ imprimitivity bimodule $X^m$ onto $X$.

\begin{rem}
Caution: \defnref{morita functor} is not a direct analogue of the definition of Morita compatibility in 
\cite[Definition~3.2]{bgwexact}, but it suits our purposes in working with coaction functors, as we will see in \propref{compose}.
\end{rem}

\begin{rem}
\lemref{Xm} says in particular
that maximalization preserves Morita equivalence of coactions.
This
is almost new: it also follows from
first applying the cross-product functor,
noting that the dual actions are ``weakly proper $G\rtimes G$-algebras'' in the sense of \cite{BusEch},
then applying \cite[Corollary~4.6]{BusEch2} with the universal crossed-product norm (denoted by $u$ in \cite{BusEch}).
\end{rem}

\begin{lem}\label{X tau}
A coaction functor $\tau$ is Morita compatible if and only if 
whenever $(X,\zeta)$ is an $(A,\delta)-(B,\epsilon)$ imprimitivity bimodule,
there are an $A^\tau-B^\tau$ imprimitivity bimodule $X^\tau$
and a $q^\tau_A-q^\tau_B$ compatible imprimitivity-bimodule homomorphism $q^\tau_X:X^m\to X^\tau$.
\end{lem}

\begin{proof}
Given $X^\tau$ and $q^\tau_X$ with the indicated properties,
by \cite[Lemma~1.20]{enchilada} we have
\[
X^m\dashind B_\tau^m
=X^m\dashind \ker q^\tau_B
=\ker q^\tau_A
=A^m_\tau.
\]
It follows that $\tau$ is Morita compatible.

Conversely, suppose $\tau$ is Morita compatible,
and let $(X^m,\zeta^m)$ be as above.
Then by the Rieffel correspondence,
$X^\tau:=X^m/X^m\cdot B_\tau^m$
is an $A^m/A^m_\tau-B^m/B^m_\tau$ imprimitivity bimodule,
and the quotient map
$q^\tau_X:X^m\to X^\tau$
is compatible with the quotient maps $A^m\mapsto A^m/A^m_\tau$ and $B^m\mapsto B^m_\tau$.
Via the unique isomorphisms making the diagrams
\[
\begin{aligned}
\xymatrix{
A^m \ar[d]_{\txt{quotient\\map}} \ar[dr]^{q^\tau_A}
\\
A^m/A^m_\tau \ar@{-->}[r]_-\simeq
&A^\tau
}
&
\xymatrix{
B^m \ar[d]_{\txt{quotient\\map}} \ar[dr]^{q^\tau_B}
\\
B^m/B^m_\tau \ar@{-->}[r]_-\simeq
&B^\tau
}
\end{aligned}
\]
commute, $q^\tau_X$ becomes $q^\tau_A-q^\tau_B$ compatible.
\end{proof}

\begin{ex}
It follows trivially that the maximalization functor is Morita compatible.
\end{ex}

\begin{lem}\label{id}
The identity functor on coactions is Morita compatible.
\end{lem}

\begin{proof}
Let $(X,\zeta)$ be an $(A,\delta)-(B,\epsilon)$ imprimitivity bimodule,
and let $(X^m,\zeta^m)$ be the associated $(A^m,\delta^m)-(B^m,\epsilon^m)$ imprimitivity bimodule from \lemref{Xm}.
By 
\lemref{X tau}
it suffices to find a
$q^m_A-q^m_B$ compatible
imprimitivity-bimodule homomorphism
$q^m_X:X^m\to X$.
Now, $X^m$ is the upper right corner of the $2\times 2$ matrix representation of the linking algebra $L^m$,
and the maximalization map $q^m_L$ of the linking algebra $L$ of $X$ preserves the upper right corners. Thus $q^m_L$ takes $X^m$ onto $X$, and simple algebraic manipulations show that it has the right properties.
\end{proof}

\begin{thm}\label{glb exact}
The greatest lower bound of the collection of all exact and Morita compatible coaction functors is itself exact and Morita compatible.
\end{thm}

\begin{proof}
Let $\TT$ be the collection of all exact and Morita compatible coaction functors,
and let $\tau$ be the greatest lower bound of $\TT$.
As in the proof of \thmref{glb},
for every coaction $(A,\delta)$ we have
\[
A^m_\tau=\clspn_{\sigma\in\TT}A^m_\sigma.
\]
For exactness, we apply \thmref{exact functor}
Let
\[
\xymatrix{
0\ar[r]
&(I,\gamma) \ar[r]^-\phi
&(A,\delta) \ar[r]^-\psi
&(B,\epsilon)
\ar[r]&0
}
\]
be a short exact sequence of coactions.
Then
\begin{align*}
\phi^m(I^m_\tau)
&=\phi^m\left(\clspn_{\sigma\in\TT}I^m_\sigma\right)
\\&=\clspn_{\sigma\in\TT}\phi^m(I^m_\sigma)
\\&=\clspn_{\sigma\in\TT}\bigl(\phi^m(I^m)\cap A^m_\sigma\bigr)
\righttext{(since $\sigma$ is exact)}
\\&=\phi^m(I^m)\cap \clspn_{\sigma\in\TT}A^m_\sigma
\\&\hspace{.5in}\text{(since all spaces involved are ideals in $C^*$-algebras)}
\\&=\phi^m(I^m)\cap A^m_\tau,
\end{align*}
and
\begin{align*}
\phi^m(I^m)+A^m_\tau
&=\phi^m(I^m)+\clspn_{\sigma\in\TT}A^m_\sigma
\\&=\clspn_{\sigma\in\TT}\bigl(\phi^m(I^m)+A^m_\sigma\bigr)
\\&=\clspn_{\sigma\in\TT}(\psi^m)\inv(B^m_\sigma)
\righttext{(since $\sigma$ is exact)}
\\&=(\psi^m)\inv\bigl(\clspn_{\sigma\in\TT}B^m_\sigma\bigr)
\\&=(\psi^m)\inv(B^m_\tau),
\end{align*}
so $\tau$ is exact.

For Morita compatibility,
let $(X,\zeta)$ be an $(A,\delta)-(B,\epsilon)$ imprimitivity bimodule,
with associated $A^m-B^m$ imprimitivity bimodule $X^m$.
Then
\begin{align*}
X^m\dashind B^m_\tau
&=X^m\dashind \clspn_{\sigma\in\TT}B^m_\sigma
\\&=\clspn_{\sigma\in\TT} X^m\dashind B^m_\sigma
\\&\hspace{.5in}\righttext{(by continuity of Rieffel induction)}
\\&=\clspn_{\sigma\in\TT} A^m_\sigma\righttext{(since $\sigma$ is Morita compatible)}
\\&=A^m_\tau,
\end{align*}
so $\tau$ is Morita compatible.
\end{proof}

\begin{defn}
We call the above greatest lower bound of the collection of all exact and Morita compatible coaction functors
the \emph{minimal exact and Morita compatible coaction functor}.
\end{defn}

\subsection*{Comparison with \cite{bgwexact}}

As we mentioned previously, \cite[see page~8]{bgwexact} defines one crossed-product functor $\sigma_1$ to be \emph{smaller} than another one $\sigma_2$, written $\sigma_1\le \sigma_2$, if the natural surjection $A\rtimes_{\alpha,\sigma_2} G\to A\rtimes_{\alpha,r} G$ factors through the $\sigma_1$-crossed product.

Let $\tau$ be a coaction functor, and let $\sigma=\tau\circ\cp$ be the associated crossed-product functor, i.e.,
\[
(A,\alpha)^\sigma=A\rtimes_{\alpha,\sigma} G:=(A\rtimes_\alpha G)^\tau.
\]
For a morphism $\phi:(A,\alpha)\to (B,\beta)$ of actions,
we write
\[
\phi\rtimes_\sigma G=(\phi\rtimes G)^\tau:A\rtimes_{\alpha,\sigma} G\to B\rtimes_{\beta,\sigma} G
\]
for the associated morphism of $\sigma$-crossed products.

\begin{prop}\label{compose}
With the above notation, if the coaction functor $\tau$ is exact or Morita compatible, then the associated crossed-product functor $\sigma$ has the same property.
Moreover, if $\tau_1\le \tau_2$ then $\sigma_1\le \sigma_2$.
\end{prop}

\begin{proof}
The last statement follows immediately from the definitions.
For the other statement, first assume that $\tau$ is exact,
and let
\[
\xymatrix{
0\ar[r]&(I,\gamma)\ar[r]^\phi &(A,\alpha)\ar[r]^\psi &(B,\beta)\ar[r]&0
}
\]
be a short exact sequence of actions.
Then the sequence
\[
\xymatrix{
0\ar[r] &(I\rtimes_\gamma G,\what\gamma) \ar[r]^{\phi\rtimes G} &(A\rtimes_\alpha G,\what\alpha) \ar[r]^{\psi\rtimes G} &(B\rtimes_\beta G,\what\beta)\ar[r]&0
}
\]
of coactions is exact,
since the full-crossed-product functor is exact.
Then by exactness of $\tau$ we see that the sequence
\[
\xymatrix@C+10pt{
0\ar[r] &I\rtimes_{\gamma,\sigma} G \ar[r]^{\phi\rtimes_\sigma G} & A\rtimes_{\alpha,\sigma} G \ar[r]^{\psi\rtimes_\sigma G} & B\rtimes_{\beta,\sigma} G \ar[r] & 0
}
\]
is also exact.

On the other hand, assume that the coaction functor $\tau$ is Morita compatible.
As in \cite[Section~3]{bgwexact}, the \emph{unwinding isomorphism} $\Phi$,
which is the integrated form of the covariant pair
\begin{align*}
\pi(a\otimes k)&=i_A(a)\otimes k\\
u_s&=i_G(s)\otimes \lambda_s,
\end{align*}
fits into a diagram
\begin{equation}\label{unwind}
\xymatrix@C+20pt{
(A\otimes\KK)\rtimes_{\alpha\otimes\ad\lambda} G \ar[r]^-\Phi_-\simeq
\ar[d]_{q^\tau_{(A\otimes\KK)\rtimes_{\alpha\otimes\ad\lambda} G}}
&(A\rtimes_\alpha G)\otimes\KK \ar[d]^{q^\tau_{A\rtimes_\alpha G}\otimes\id}
\\
(A\otimes\KK)\rtimes_{\alpha\otimes\ad\lambda,\sigma} G
\ar@{-->}[r]^-\simeq_-\Upsilon
&(A\rtimes_{\alpha,\sigma} G)\otimes\KK,
}
\end{equation}
i.e.,
\[
\ker q^\tau_{(A\otimes\KK)\rtimes_{\alpha\otimes\ad\lambda} G}
=\ker (q^\tau_{A\rtimes_\alpha G}\otimes\id)\circ \Phi.
\]
The diagram \eqref{unwind} fits into a more elaborate diagram
\[
\xymatrix@C+5pt{
(A\otimes\KK)\rtimes_{\alpha\otimes\ad\lambda} G \ar[r]^-\Phi_-\simeq
\ar[d]_{q^\tau_{(A\otimes\KK)\rtimes_{\alpha\otimes\ad\lambda} G}}
&(A\rtimes_\alpha G)\otimes\KK
\ar[d]|{q^\tau_{(A\rtimes_\alpha G)\otimes\KK}}
\ar@/^1pc/[ddr]^{q^\tau_{A\rtimes_\alpha G}\otimes\id}
\\
(A\otimes\KK)\rtimes_{\alpha\otimes\ad\lambda,\sigma} G
\ar@{-->}[r]^-\simeq_-{\Phi^\tau}
\ar@{-->}@/_1pc/[drr]_\Upsilon^\simeq
&((A\rtimes_\alpha G)\otimes\KK)^\tau \ar@{-->}[dr]^\theta_\simeq
\\
&&(A\rtimes_{\alpha,\sigma} G)\otimes\KK,
}
\]
which we proceed to analyze.
There is a unique
\[
\bigl(\what{\alpha\otimes\ad\lambda}\bigr)^\tau-(\what\alpha\otimes_*\id)^\tau
\]
equivariant
homomorphism $\Phi^\tau$
making the upper-left rectangle commute,
since $\tau$ is functorial.
Moreover, $\Phi^\tau$ is an isomorphism since $\Phi$ is, again by functoriality.
Applying Morita compatibility of $\tau$ to the equivariant
$((A\rtimes_\alpha G)\otimes\KK)-(A\rtimes_\alpha G)$
imprimitivity bimodule
$(A\rtimes_\alpha G)\otimes L^2(G)$
shows that there is a unique 
\[
(\what\alpha\otimes_*\id)^\tau-(\what\alpha^\tau\otimes_*\id)
\]
equivariant
isomorphism $\theta$
that makes the upper right triangle commute.
Thus there is a unique isomorphism $\Upsilon$ making the lower left triangle commute,
and then the outer quadrilateral commutes, as desired.
\end{proof}

\begin{q}\

\begin{enumerate}
\item
Is the minimal exact and Morita compatible crossed product of \cite[Section~4]{bgwexact} naturally isomorphic to the composition of the minimal exact and Morita compatible coaction functor and the full crossed product?

\item
More generally, given a crossed-product functor on actions, when does it decompose as a full crossed product followed by a coaction functor?
Does it make any difference if the crossed-product functor is exact or Morita compatible?
\end{enumerate}
\end{q}

\section{Decreasing coaction functors}\label{decreasing}

In this section we introduce a particular type of coaction functor with the convenient property that we do not need to check things by going through the maximalization functor, as we'll see in Propositions~\ref{decreasing exact} and \ref{decreasing morita}.
Suppose that for each coaction $(A,\delta)$ we have a coaction $(A^\tau,\delta^\tau)$ and a $\delta-\delta^\tau$ equivariant surjection $Q^\tau:A\to A^\tau$,
and further suppose that for each morphism $\phi:(A,\delta)\to (B,\epsilon)$ we have
\[
\ker Q^\tau_A\subset \ker Q^\tau_B \circ \phi,
\]
so that there is a unique morphism $\phi^\tau$ making the diagram
\[
\xymatrix{
(A,\delta) \ar[r]^-\phi \ar[d]_{Q^\tau_A}
&(B,\epsilon) \ar[d]^{Q^\tau_B}
\\
(A^\tau,\delta^\tau) \ar@{-->}[r]_-{\phi^\tau}^-{!}
&(B^\tau,\delta^\tau)
}
\]
commute.
The uniqueness and surjectivity assumptions imply that $\tau$ constitutes a functor on the category of coactions, and moreover $Q^\tau:\id\to \tau$ is a natural transformation.

\begin{defn}\label{decreasing defn}
We call a functor $\tau$ as above \emph{decreasing} if for each coaction $(A,\delta)$ we have
\[
\ker Q^\tau_A\subset \ker \Lambda_A.
\].
\end{defn}

\begin{lem}\label{dec coact}
Every decreasing functor $\tau$ on coactions is a coaction functor, and moreover $\tau\le \id$.
\end{lem}

\begin{proof}
For each coaction $(A,\delta)$,
define a homomorphism $q^\tau_A$ by the commutative diagram
\[
\xymatrix{
A^m \ar[d]_{q^m_A} \ar[dr]^{q^\tau_A}
\\
A \ar[r]_-{Q^\tau_A}
&A^\tau,
}
\]
where $q^m_A$ is the maximalization map.
$q^\tau$ is natural and surjective since both $q^m$ and $Q^\tau$ are.
We have
\begin{align*}
\ker q^\tau_A
&=\{a\in A^m:q^m_A(a)\in \ker Q^\tau_A\}
\\&\subset \{a\in A^m:q^m_A(a)\in \ker \Lambda_A\}
\\&=\ker \Lambda_A\circ q^m_A
\\&=\ker \Lambda_{A^m}.
\end{align*}
Thus $\tau$ is a coaction functor, and then $\tau\le \id$ by \lemref{order}.
\end{proof}

\begin{notn}
For a decreasing coaction functor $\tau$ and any coaction $(A,\delta)$ put
\[
A_\tau=\ker Q^\tau_A.
\]
\end{notn}

\begin{prop}\label{decreasing exact}
A decreasing coaction functor $\tau$ is exact if and only if for any short exact sequence
\begin{equation}\label{seq dec}
\xymatrix{
0\ar[r] &(I,\delta_I)\ar[r]^\phi &(A,\delta)\ar[r]^\psi &(B,\delta^I)\ar[r]&0
}
\end{equation}
of coactions,
both
\[
\phi(I_\tau)=\phi(I)\cap A_\tau
\]
and
\[
\phi(I)+A_\tau\supset \psi\inv(B_\tau)
\]
hold.
\end{prop}

\begin{proof}
The proof is very similar to, and slightly easier than, that of \thmref{functor exact}, using the commutative diagram
\[
\xymatrix{
&0\ar[d]&0\ar[d]&0\ar[d]
\\
0\ar[r]&I_\tau\ar[d]_{\iota_I}\ar[r]^{\phi|}&A^m_\tau\ar[d]_{\iota_A}\ar[r]^{\psi|}&B^m_\tau\ar[d]_{\iota_{B}}\ar[r]&0
\\
0\ar[r]&I\ar[d]_{Q^\tau_I}\ar[r]^{\phi}&A\ar[d]_{Q^\tau_A}\ar[r]^{\psi}&B\ar[d]_{Q^\tau_{B}}\ar[r]&0
\\
0\ar[r]&I^\tau\ar[r]^{\phi^\tau}\ar[d]&A^\tau\ar[r]^{\psi^\tau}\ar[d]&B^\tau\ar[r]\ar[d]&0
\\
&0&0&0.
}
%\qedhere
\]
\end{proof}

\begin{prop}\label{decreasing morita}
A decreasing coaction functor $\tau$ is Morita compatible if and only if
whenever $(X,\zeta)$ is an $(A,\delta)-(B,\epsilon)$ imprimitivity bimodule,
there are an $A^\tau-B^\tau$ imprimitivity bimodule $X^\tau$ and a $Q^\tau_A-Q^\tau_B$ compatible imprimitivity-bimodule homomorphism $Q^\tau_X:X\to X^\tau$.
\end{prop}

\begin{proof}
First suppose $\tau$ is Morita compatible.
Let $(X,\zeta)$ be an $(A,\delta)-(B,\epsilon)$ imprimitivity bimodule,and let $q^\tau_X:X^m\to X^\tau$
be a $q^m_A-q^m_B$ compatible imprimitivity-bimodule homomorphism onto an $A^\tau-B^\tau$ imprimitivity bimodule $X^\tau$, as in \lemref{X tau}.
By Lemmas~\ref{id} and \ref{X tau} there is also a $q^m_A-q^m_B$ compatible imprimitivity bimodule homomorphism $q^m_X$ of $X^m\to X$.
By definition, we have
\[
q^\tau_A=Q^\tau_A\circ q^m_A:A^m\to A^\tau.
\]
Thus
\begin{align*}
\ker q^m_X
&=(\ker q^m_A)\cdot X^m
\\&\subset (\ker Q^\tau_A\circ q^m_A)\cdot X^m
\\&=(\ker q^\tau_A)\cdot X^m
\\&=\ker q^\tau_X,
\end{align*}
and hence $q^\tau_X$ factors through a commutative diagram
\[
\xymatrix{
X^m \ar[dd]_{q^\tau_X} \ar[dr]^{q^m_X}
\\
&X \ar@{-->}[dl]^{Q^\tau_X}_{!}
\\
X^\tau
}
\]
for a unique imprimitivity bimodule homomorphism $Q^\tau_X$.
Moreover, $Q^\tau_X$ is compatible on the left with $Q^\tau_A$ by construction,
and similar reasoning, using the Rieffel correspondence of ideals, shows that it is also $Q^\tau_B$ compatible on the right.

Conversely, suppose we have $(X,\zeta)$, $X^\tau$, and $Q^\tau_X$ as indicated,
and let $(X^m,\zeta^m)$ be the associated $(A^m,\delta^m)-(B^m,\epsilon^m)$ imprimitivity bimodule from \lemref{Xm}.
By \lemref{X tau} it suffices to find a $q^m_A-q^m_B$ compatible imprimitivity-bimodule homomorphism $q^\tau_X:X^m\to X^\tau$.
Since $q^\tau=Q^\tau\circ q^m$ on both $A^m$ and $B^m$,
by \lemref{id} and our assumptions we can take
$q^\tau_X=Q^\tau_X\circ q^m_X$.
\end{proof}

\section{Coaction functors from large ideals}\label{large}

The most important source of examples of the decreasing coaction functors of the preceding section is large ideals.
We recall some basic concepts from \cite{graded, exotic}.
Let $E$ be an ideal of $B(G)$ that is \emph{large},
meaning it is nonzero, $G$-invariant, and weak* closed.
Then the preannihilator $\pann E$ of $E$ in $C^*(G)$ is an ideal contained in the kernel of the regular representation $\lambda$.
Write $C^*_E(G)=C^*(G)/\pann E$ for the quotient group $C^*$-algebra
and
$q_E:C^*(G)\to C^*_E(G)$ for the quotient map.
The ideal $\pann E=\ker q_E$ of $C^*(G)$
is \emph{weakly} $\delta_G$-invariant,
i.e.,
$\delta_G$ descends to a coaction, which we denote by $\delta_G^E$, on the quotient $C^*_E(G)$.

For any coaction $(A,\delta)$ and any large ideal $E$ of $B(G)$,
\[
A_E:=\{a\in A:E\cdot a=\{0\}\}=\ker(\id\otimes q_E)\circ\delta
\]
is a \emph{small} ideal of $A$ (that is, an ideal contained in $\ker j_A=\ker \Lambda_A$)
and we write
$A^E=A/A_E$ for the quotient $C^*$-algebra and
$Q^E_A:A\to A^E$ for the quotient map.
$A_E$ is weakly $\delta$-invariant \cite[Lemma~3.5]{exotic},
and we write $\delta^E$ for the quotient coaction on $A^E$.

\begin{rem}
The properties of the $B(G)$-module structure (see \appxref{module lemmas}) allow for a shorter proof of invariance than in \cite{exotic}:
if $a\in A_E$, $f\in B(G)$, and $g\in E$ then
\[
g\cdot (f\cdot a)=(gf)\cdot a=0,
\]
because $E$ is an ideal, and it follows that $B(G)\cdot A_E\subset A_E$.
\end{rem}

\begin{prop}\label{E coaction functor}
$(A,\delta)\mapsto (A^E,\delta^E)$ is a decreasing coaction functor,
which we denote by $\tau_E$.
\end{prop}

\begin{proof}
By the above discussion and \lemref{dec coact}, it suffices to observe that for any morphism $\phi:(A,\delta)\to (B,\epsilon)$ of coactions 
and
for all $a\in \ker Q^E_A$ and $f\in E$,
\begin{align*}
f\cdot \phi(a)
&=\phi(f\cdot a)=0,
\end{align*}
which implies that $\ker Q^E_A\subset \ker Q^E_B\circ \phi$.
\end{proof}

\begin{rem}
The above lemma should be compared with
\cite[Corollary~6.5 and Lemma~7.1]{BusEch}, 
\cite[Lemma~2.3]{BusEch2}, and
\cite[Lemma~A.3]{bgwexact}.
\end{rem}

\begin{ex}
$\tau_{B(G)}$ is the identity functor.
\end{ex}

\begin{ex}
$\tau_{B_r(G)}$ is naturally isomorphic to the normalization functor.
\end{ex}

\begin{ex}
The maximalization functor is not of the form $(A,\delta)\mapsto (A^E,\delta^E)$ for any large ideal $E$ of $B(G)$,
because the maximalization functor is not decreasing in the sense of \defnref{decreasing defn}.
\end{ex}

\begin{prop}\label{exact E functor}
For a large ideal $E$ of $B(G)$,
the coaction functor $\tau_E$ is exact if and only if
for every coaction $(A,\delta)$ and
every strongly invariant ideal $I$ of $A$,
\begin{equation}\label{test exact E}
I+A_E\supset \{a\in A:E\cdot a\subset I\}.
\end{equation}
\end{prop}

\begin{proof}
Let
\begin{equation}\label{seq}
\xymatrix{
0\ar[r] & (I,\zeta) \ar[r]^\phi & (A,\delta) \ar[r]^\psi & (B,\epsilon) \ar[r] &0
}
\end{equation}
be a short exact sequence of coactions.
Exactness of the associated sequence
\begin{equation}\label{E seq}
\xymatrix{
0\ar[r] & I^E \ar[r]^{\phi^E} & A^E \ar[r]^{\psi^E} & B^E \ar[r] &0
}
\end{equation}
will not be affected if we replace the short exact sequence \eqref{seq} by an isomorphic one, so without loss of generality 
$\phi$ is the inclusion of an ideal $I$ of $A$
and $\psi$ is the quotient map onto $B=A/I$.
By \propref{decreasing exact},
the sequence \eqref{E seq} is exact if and only if
\begin{equation}\label{IE}
I_E=I\cap A_E
\end{equation}
and
\begin{equation}\label{AE}
I+A_E\supset \psi\inv(B_E).
\end{equation}
Since
\[
I_E=\{a\in I:E\cdot a=\{0\}\},
\]
\eqref{IE} automatically holds in this context.
On the other hand, \eqref{AE} is equivalent to \eqref{test exact E} because
\begin{align*}
B_E
&=\{a+I\in B=A/I:E\cdot (a+I)=\{0\}\}
\\&=\{a+I:E\cdot a\subset I\}.
\qedhere
\end{align*}
\end{proof}

\begin{rem}
Techniques similar those used in the above proof, showing that \eqref{IE} holds automatically,
can also be used to show that
the functor $\tau_E$ preserves injectivity of morphisms:
if $\phi:A\to B$ is an injective equivariant homomorphism
and $a\in \ker \phi^E$,
then we can write $a=Q^E_A(a')$ for some $a'\in A$.
We have
\[
0=\phi^E(a)=\phi^E\circ Q^E_A(a')=Q^E_B\circ \phi(a'),
\]
so
\[
\phi(a)\in \ker Q^E_B=B_E.
\]
Thus for all $f\in E$ we have
\[
0=f\cdot \phi(a')=\phi(f\cdot a'),
\]
so $f\cdot a'=0$ since $\phi$ is injective.
But then $a'\in A_E=\ker Q^E_A$, so $a=0$.
This remark should be compared with \cite[Proposition~6.2]{BusEch}.
\end{rem}

\begin{cor}\label{intersect}
Let $E$ and $F$ be large ideals of $B(G)$,
and let $\<EF\>$ denote the weak*-closed linear span of the set $EF$ of products.
If $\tau_E$ or $\tau_F$ is exact then $\<EF\>=E\cap F$.
\end{cor}

\begin{proof}
Without loss of generality assume that $\tau_E$ is exact.
Note that, since $E$ is an ideal of $B(G)$,
\[
\pann E=\{a\in C^*(G):E\cdot a=\{0\}\},
\]
and similarly for $\pann F$.
Claim:
\[
\pann E+\pann F=\pann \<EF\>.
\]
To see this, note that, since $E$ is exact, by \propref{exact E functor}
with 
$(A,\delta)=(C^*(G),\delta_G)$ and
$I=\pann F$
we have
\[
\pann F+\pann E\supset \{a\in C^*(G):E\cdot a\subset \pann F\}.
\]
Now, for $a\in C^*(G)$ we have
\begin{align*}
E\cdot a\subset \pann F
&\iff F\cdot (E\cdot a)=\{0\}
\\&\iff (EF)\cdot a=\{0\}
\\&\overset{*}{\iff} \<EF\>\cdot a=\{0\}
\\&\iff a\in \pann \<EF\>,
\end{align*}
where the equivalence at * holds since
for every $a\in C^*(G)$ the map from $B(G)$ to $C^*(G)$ defined by
$f\mapsto f\cdot a$
is weak*-weak continuous.
Thus $\pann F+\pann E\supset \pann \<EF\>$.

For the reverse containment, note that
$EF\subset E$ because $E$ is an ideal,
so $\<EF\>\subset E$ because $E$ is weak*-closed,
and hence $\pann E\subset \pann \<EF\>$.
Similarly, $\pann F\subset \pann \<EF\>$, and so $\pann E+\pann F\subset \pann \<EF\>$,
proving the claim.

Now,
since $\pann E$ and $\pann F$ are closed ideals of $C^*(G)$,
it follows from the elementary duality theory for Banach spaces that
\[
\pann E+\pann F=\pann (E\cap F),
\]
and the corollary follows upon taking annihilators.
\end{proof}

The following result should be compared with 
\cite[Lemma~A.5]{bgwexact}:

\begin{prop}\label{morita}
The coaction functor $\tau_E$ is Morita compatible.
\end{prop}

\begin{proof}
Let $(X,\zeta)$ be an $(A,\delta)-(B,\epsilon)$ imprimitivity bimodule.
Since $\tau$ is decreasing, by \lemref{decreasing morita}, it suffices to show that
$X\dashind B_E=A_E$.
The external tensor product $X\otimes C^*_E(G)$ is an $(A\otimes C^*_E(G))-(B\otimes C^*_E(G))$ imprimitivity bimodule, and we have an $(\id_A\otimes q_E)-(\id_B\otimes q_E)$ compatible imprimitivity bimodule homomorphism
\[
\id_X\otimes q_E:X\otimes C^*(G)\to X\otimes C^*_E(G).
\]
The composition
\[
(\id_X\otimes q_E)\circ\zeta:X\to M(X\otimes C^*_E(G))
\]
is an
$(\id_A\otimes q_E)\circ\delta-(\id_B\otimes q_E)\circ\epsilon$
compatible imprimitivity bimodule homomorphism.
We have
\begin{align*}
\ker (\id_A\otimes q_E)\circ\delta&=A_E\\
\ker (\id_B\otimes q_E)\circ\epsilon&=B_E.
\end{align*}
Thus, by \cite[Lemma~1.20]{enchilada}, $A_E$ is the ideal of $A$ associated to the ideal $B_E$ of $B$ via the Rieffel correspondence.
\end{proof}

\begin{rem}
\propref{morita} subsumes \cite[Lemma~4.8]{exotic}, which is the special case of exterior equivalent coactions.
It is tempting to try to use this to simplify the proof of \cite[Theorem~4.6]{exotic}, which says that $(A,\delta)$ satisfies $E$-crossed-product duality if and only if it is isomorphic to $({A^m}^E,{\delta^m}^E)$,
since we have Morita equivalences
\[
(A^m,\delta^m)\sim_M
(A^m\otimes \KK,\delta\otimes_*\id)\sim_M
(A\rtimes_\delta G\rtimes_{\what\delta} G,\what{\what\delta}).
\]
However, it turns out that appealing to \propref{morita} would not shorten the proof of \cite[Theorem~4.6]{exotic} much.
Nevertheless, it is interesting to note that,
by \propref{morita},
we have
\begin{align*}
(A,\delta)=({A^m}^E,{\delta^m}^E)
&\iff
(A\otimes\KK,\delta\otimes_*\id)=((A^m\otimes\KK)^E,(\delta^m\otimes_*\id)^E),
\end{align*}
equivalently
\[
\ker\Phi=(A\rtimes_\delta G\rtimes_{\what\delta} G)_E,
\]
which by definition is equivalent to $E$-crossed-product duality for $(A,\delta)$.
\end{rem}

For some purposes,
albeit not for
the purposes of this paper, a more appropriate coaction functor associated to $E$ is the following (see also \cite[Theorem~5.1]{BusEch}):

\begin{defn}
The \emph{$E$-ization} of a coaction $(A,\delta)$ is
\[
(A^\ize,\delta^\ize):=\bigl((A^m)^E,(\delta^m)^E\bigr).
\]
\end{defn}

$E$-ization is 
a functor on the category of coactions,
being
the composition of the functors maximalization and $\tau_E$.
The $E$-ization of a $\delta-\epsilon$ equivariant homomorphism $\phi:A\to B$
is
\[
\phi^\ize=(\phi^m)^E:A^{mE}\to B^{mE}.
\]

\begin{prop}
$E$-ization is a coaction functor.
\end{prop}

\begin{proof}
We must produce a suitable natural transformation $q^{\ize}:(A^m,\delta^m)\to (A^\ize,\delta^\ize)$,
and we take
\[
q^{\ize}_A=Q^E_{A^m}:A^m\to A^{mE}=A^\ize.
\]
$q^{\ize}$ is natural
since $\tau_E$ is a decreasing coaction functor.
\end{proof}

\begin{thm}\label{E morita}
For any large ideal $E$ of $B(G)$, the $E$-ization coaction functor is Morita compatible,
\end{thm}

\begin{proof}
Let $(X,\zeta)$ be an $(A,\delta)-(B,\epsilon)$ imprimitivity bimodule,
with associated $(A^m,\delta^m)-(B^m,\epsilon^m)$ imprimitivity bimodule $(X^m,\zeta^m)$.
We must show that
\[
X^m\dashind \ker q^{\ize}_B=\ker q^{\ize}_A.
\]
But this follows immediately by applying \propref{morita} to $(X^m,\zeta^m)$,
since
$q^{\ize}_A=Q_E^{A^m}$ and
$q^{\ize}_B=Q_E^{B^m}$.
\end{proof}

\begin{rem}\label{mE}
For any large ideal $E$,
the two coaction functors $\tau_E$ and $E$-ization have similar properties,
e.g., they are both Morita compatible
(\propref{morita} and \thmref{E morita}).
However, in general they are not naturally isomorphic functors. For example, if $E=B(G)$ then 
$\tau_E$ is the identity functor and $E$-ization is maximalization.
That being said, for $E=B_r(G)$ we do have $\tau_E\cong \tau_E\circ \text{maximalization}$.
\end{rem}

Note that, given a coaction $(A,\delta)$, we have two homomorphisms of the maximalization $(A^m,\delta^m)$:
\[
\xymatrix{
(A^m,\delta^m) \ar[d]^{q^m} \ar[dr]^{q^{\ize}}
\\
(A,\delta)
&(A^\ize,\delta^\ize);
}
\]
in \cite[Definition~3.7]{graded} we said $(A,\delta)$ is \emph{$E$-determined from its maximalization} if $\ker q^m=\ker q^{\ize}$, in which case there is a natural isomorphism $(A,\delta)\cong (A^\ize,\delta^\ize)$.

Given an action $(B,\alpha)$, in \cite[Definition~6.1]{graded} we defined the \emph{$E$-crossed product} as
\[
B\rtimes_{\alpha,E} G=(B\rtimes_\alpha G)/(B\rtimes_\alpha G)_E=(B\rtimes_\alpha G)^E,
\]
where in the last expression we have composed the full-crossed-product functor with $\tau_E$.

As in \cite[Definition~4.5]{BusEch}, we say a coaction $(A,\delta)$ satisfies \emph{$E$-duality} (called ``$E$-crossed product duality'' in \cite[Definition~4.3]{exotic}), or is an \emph{$E$-coaction}, if
there is an isomorphism $\theta$ making the diagram
\[
\xymatrix{
A\rtimes_\delta G\rtimes_{\what\delta} G \ar[r]^-\Phi \ar[d]_{Q_E}
&A\otimes\KK
\\
A\rtimes_\delta G\rtimes_{\what\delta,E} G \ar@{-->}[ur]_\theta^\simeq
}
\]
commute,
equivalently
\[
\ker\Phi=(A\rtimes_\delta G\rtimes_{\what\delta} G)_E,
\]
where $\Phi$ is the canonical surjection.

In \cite[Theorem~4.6]{exotic} 
we proved that $(A,\delta)$ is an $E$-coaction if and only if it is $E$-determined from its maximalization.
(\cite[Theorem~5.1]{BusEch} proves the converse direction.)

\begin{lem}\label{determined}
For a coaction $(A,\delta)$, the following are equivalent:
\begin{enumerate}
\item
$(A,\delta)$ is an $E$-coaction

\item
$(A,\delta)$ is $E$-determined from its maximalization.

\item
There exists a maximal coaction $(B,\epsilon)$ such that $(A,\delta)\cong (B^E,\epsilon^E)$.
\end{enumerate}
\end{lem}

\begin{proof}
The equivalence of (1) and (2) is
\cite[Theorem~4.6]{exotic}, and
(2) trivially implies (3).
Assume (3), i.e., that $(B,\epsilon)$ is maximal and we have an isomorphism $\theta:(B^E,\epsilon^E)\to (A,\delta)$.
The surjection $Q_E^B:(B,\epsilon)\to (B^E,\epsilon^E)$ is a maximalization,
since $\epsilon$ is maximal and $\ker Q_E^B\subset \ker q^n_B$.
Thus $\theta\circ Q_E^B$ is a maximalization of $(A,\delta)$.
Since any two maximalizations of $(A,\delta)$ are isomorphic,
there is an isomorphism $\psi$ making the diagram
\[
\xymatrix{
(A^m,\delta^m) \ar[d]_{q^m_A}
&(B,\epsilon) \ar@{-->}[l]_-\psi^-\simeq \ar[d]^{Q_E}
\\
(A,\delta)
&(B^E,\epsilon^E) \ar[l]^-\theta_-\simeq
}
\]
commute.
Thus $q^m_A\circ\psi$ is also a maximalization of $(A,\delta)$.
Therefore
\begin{align*}
\ker q^m_A
&=\psi\bigl(\ker Q_E\bigr)
\\&=\psi\bigl(B_E\bigr)
\\&=A^m_E,
\end{align*}
giving (2).
\end{proof}

\begin{thm}\label{equivalence}
The functor $\tau_E$
restricts 
to give an equivalence of the category of maximal coactions to the category of $E$-coactions.
\end{thm}

Note: in the above statement, we mean the \emph{full} subcategories of the category of coactions.

\begin{proof}
By abstract nonsense, it suffices to show that the functor is essentially surjective and fully faithful, i.e.,
\begin{enumerate}
\item
every $E$-coaction $(A,\delta)$ is isomorphic to $(B^E,\epsilon^E)$ for some maximal coaction $(B,\epsilon)$, and

\item
for any two maximal coactions $(A,\delta)$ and $(B,\epsilon)$,
\[
\phi\mapsto \phi^E
\]
maps the set of equivariant homomorphisms $\phi:A\to B$
bijectively onto the set of equivariant homomorphisms $\psi:A^E\to B^E$.
\end{enumerate}

(1) is immediate from \lemref{determined}.
For (2),
given maximal coactions $(A,\delta)$ and $(B,\epsilon)$
and distinct nondegenerate equivariant homomorphisms $\phi,\psi:A\to B$,
we have an equivariant commutative diagram
\[
\xymatrix{
A \ar[rr]^-\phi \ar[dr]^{Q^E_A} \ar[dd]_{\Lambda_A}
&&B \ar[dr]^{Q^E_B} \ar@{-}[d]
\\
&A^E \ar[rr]^(.3){\phi^E} \ar[dl]^{\Lambda^E_A}
&{} \ar[d]_(.4){\Lambda_B}
&B^E \ar[dl]^{\Lambda^E_B}
\\
A^n \ar[rr]_-{\phi^n}
&&B^n,
}
\]
where $Q^E_A$ is a maximalization of $(A^E,\delta^E)$, $\Lambda_A$ is a normalization of $(A,\delta)$, and $\Lambda^E_A$ is a normalization of $(A^E,\delta^E)$, and similarly for the right-hand triangle involving the $B$'s.
There is a similar commutative diagram for $\psi$.
Since the normalizations $\phi^n$ and $\psi^n$ are distinct, by
\cite[Corollary~6.1.19]{bkq1}, we must have $\phi^E\ne \psi^E$ by commutativity of the diagram.
This proves injectivity.
For the surjectivity,
let $\sigma:A^E\to B^E$ be an equivariant homomorphism.
Then the maximalization $\sigma^m:A\to B$ of $\sigma$ is the unique equivariant homomorphism making the diagram
\[
\xymatrix{
A \ar[r]^-{\sigma^m} \ar[d]_{Q^E_A}
&B \ar[d]^{Q^E_B}
\\
A^E \ar[r]_-\sigma
&B^E
}
\]
commute.
Applying the functor $\tau_E$, we see that the diagram
\[
\xymatrix{
A \ar[r]^-{\sigma^m} \ar[d]_{Q^E_A}
&B \ar[d]^{Q^E_B}
\\
A^E \ar[r]_-{(\sigma^m)^E}
&B^E
}
\]
also commutes, so we must have $\sigma^m=((\sigma^m)^E)^m$ by the universal property of maximalization,
and hence $\sigma=(\sigma^m)^E$ by \cite[Corollary~6.1.19]{bkq1} again.
\end{proof}

\begin{rem}
Much of the development in this paper regarding ``classical'' categories carries over to the ``nondegenerate'' categories (involving multiplier algebras). The nondegenerate version of the above result
resembles the ``maximal-normal equivalence'' of \cite[Theorem~3.3]{clda},
which says that normalization restricts to an equivalence between maximal and normal coactions.
However, there are some properties missing: for example,
the functor $\tau_E$ is not a reflector in the categorical sense,
because
\[
Q_E:(A^E,\delta^E)\to (A^{EE},\delta^{EE})
\]
is not an isomorphism in general.
Indeed,
\cite[Proposition~8.4]{exotic} shows that
if $(A,\delta)$ is a maximal coaction
then the composition $(\id\otimes q_E)\circ\delta^E$ in the commutative diagram
\[
\xymatrix@C+30pt{
A \ar[d]_{Q_E}
\\
A^E \ar[r]^-{\delta^E} \ar[dr]_{(\id\otimes q_E)\circ\delta^E\hspace{.2in}}
&M(A^E\otimes C^*(G)) \ar[d]^{\id\otimes q_E}
\\
&M(A^E\otimes C^*_E(G))
}
\]
need not be faithful.
Thus we cannot characterize the $E$-coactions as the coactions that are ``$E$-normal''
in the sense that the map $Q_E$ is faithful.
Furthermore, unlike with 
normalization,
\remref{mE} shows that $\tau_E$ is not isomorphic to its composition with maximalization.
\end{rem}

\begin{q}\label{min E}
Let $\FF$ be a collection of large ideals of $B(G)$,
and let
\[
F=\bigcap_{E\in\FF}E.
\]
Then $F$ is a large ideal of $B(G)$.
Is 
$\tau_F$
a greatest lower bound for the coaction functors 
$\{\tau_E:E\in\FF\}$?
(It is easy to see that 
$\tau_F$ is a lower bound.)
What if we take $\FF$ to be the set of all large ideals $E$ of $B(G)$ for which 
$\tau_E$
is exact?
\end{q}

\begin{q}
Given a coaction functor $\tau$,
is there a large ideal $E$ of $B(G)$
such that, after restricting to maximal coactions, $\tau$ is naturally isomorphic to $\tau_E$?
Note that at the level of objects the statement is false:
\cite[Example~5.4]{BusEch} gives a source of examples of a maximal coaction $(A,\delta)$ and a weakly invariant ideal $I\subset \ker q^n_A$ such that the quotient coaction $(A/I,\delta^I)$ is not of the form $(A^E,\delta^E)$ for any large ideal $E$.
(\cite[Theorem~6.10]{exotic} gives related examples, albeit not involving maximal coactions.)
Here is a related question: do there exist coaction functors that include the Buss-Echterhoff examples?
Such a functor could not be exact, since the Buss-Echterhoff examples are explicitly based upon short exact sequences whose image under the quotient maps are not exact.
We could ask the same question for the functors $\tau_E$, which, again, is exact for $E=B(G)$ but not for $E=B_r(G)$.
\end{q}

\begin{q}
For which large ideals $E$ is the coaction functor $E$-ization exact?
Exactness trivially holds for $E=B(G)$, since $B(G)$-ization coincides with maximalization.
On the other hand, exactness does not always hold for $E=B_r(G)$, because Gromov has shown the existence of nonexact groups.
\end{q}

\begin{q}
Let $\tau$ be the minimal exact and Morita compatible coaction functor.
Applying $\tau$ to the canonical coaction $(C^*(G),\delta_G)$,
we get a coaction $(C^*(G)^\tau,\delta_G^\tau)$,
with a canonical quotient map
\[
q^\tau:C^*(G)\to C^*(G)^\tau.
\]
Then
\[
E_\tau:=(\ker q^\tau)^\perp
\]
is a large ideal of $B(G)$,
by \cite[Corollary~3.13]{graded}.

Does the functor $\tau$ coincide with ${E_\tau}$-ization?
This is related to the following question:
is
\begin{align*}
E_\tau
&=\bigcap\{E:\text{$E$ is a large ideal such}
\\&\hspace{1in}\text{that $E$-ization is exact}\}?
\end{align*}
Again we could ask the analogous questions for $\tau_E$.
See also the discussion in \cite[Section~8.1]{bgwexact}.
\end{q}

\begin{rem}
Related to \qref{min E} above,
what if we consider only finitely many large ideals?
Let $E$ and $F$ be two large ideals, and let $D=E\cap F$, which is also a large ideal.
Suppose that the coaction functors $\tau_E$ and $\tau_F$ are both exact.

Is $\tau_D$ exact?
We proved in \corref{intersect} that
exactness of $E$ implies that $D$ is the weak*-closed span of the set of products $EF$,
and then we can deduce from this that if
\[
\xymatrix{
0\ar[r]
&(I,\gamma)\ar[r]^\phi
&(A,\delta)\ar[r]^\psi
&(B,\epsilon)\ar[r]
&0
}
\]
is a short exact sequence of coactions,
and if we assume that $\delta$ is \emph{w-proper} in the sense that $(\omega\otimes\id)\circ\delta(A)\subset C^*(G)$ for all $\omega\in A^*$,
then the sequence
\[
\xymatrix{
0\ar[r]
&I^D\ar[r]^{\phi^D}
&A^D\ar[r]^{\psi^D}
&B^D\ar[r]
&0
}
\]
is exact.
We see a way to parlay this into a proof that $\tau_D$ is indeed exact,
but this requires a somewhat more elaborate version of Morita compatibility,
involving not only imprimitivity bimodules but more general $C^*$-correspondences.
This will perhaps resemble the property that Buss, Echterhoff, and Willett call \emph{correspondence functoriality} (see \cite[Theorem~4.9]{BusEchWil}).
We plan to address this in a forthcoming publication.
\end{rem}

\begin{appendix}

\section{$B(G)$-module lemmas}\label{module lemmas}

Every coaction $\delta:A\to M(A\otimes C^*(G))$ gives rise to a $B(G)$-module structure on $A$ via
\[
f\cdot a=(\id\otimes f)\circ\delta(a)\righttext{for}f\in B(G),a\in A.
\]
We feel that this module structure is under-appreciated,
and will point out here several situations in which it 
makes things easier, since it allows us
to avoid computations with tensor products.

\begin{prop}\label{equivariant}
Let $(A,\delta)$ and $(B,\epsilon)$ be coactions of $G$, and let $\varphi:A\to B$ be a homomorphism. Then $\phi$ is $\delta-\epsilon$ equivariant if and only if
it is a module map, i.e., 
\[
\varphi(f\cdot a)=f\cdot \varphi(a)\righttext{for all}f\in B(G),a\in A.
\]
\end{prop}

\begin{proof}
First assume that $\phi$ is $\delta-\epsilon$ equivariant,
and let $f\in B(G)$ and $a\in A$. Then
\begin{align*}
\phi(f\cdot a)
&=\phi\bigl((\id\otimes f)\circ\delta(a)\bigr)
\\&=(\id\otimes f)\bigl((\phi\otimes\id)\circ\delta(a)\bigr)
\\&=(\id\otimes f)\bigl(\epsilon\circ\phi(a)\bigr)
\\&=f\cdot \phi(a).
\end{align*}

Conversely, assume that $\phi$ is a module map,
and let $a\in A$.
Then for every $f\in B(G)$ the above computation shows that
\[
(\id\otimes f)\bigl((\phi\otimes\id)\circ\delta(a)\bigr)
=(\id\otimes f)\bigl(\epsilon\circ\phi(a)\bigr),
\]
and it follows that $(\phi\otimes\id)\circ\delta(a)
=\epsilon\circ\phi(a)$ since slicing by $B(G)=C^*(G)^*$ separates points of $M(B\otimes C^*(G))$.
\end{proof}

\begin{prop}\label{invariant module}
Let $(A,\delta)$ be a coaction, and let $I$ be an ideal of $A$ then $I$ is weakly $\delta$-invariant if and only if it is invariant for the module structure, i.e.,
\[
B(G)\cdot I\subset I.
\]
\end{prop}

\begin{proof}
First assume that $I$ is $\delta$-invariant,
and let $f\in B(G)$ and $a\in I$.
We must show that $f\cdot a\in I$.
Let $q:A\to A/I$ be the quotient map.
We have
\begin{align*}
q(f\cdot a)
&=q\bigl((\id\otimes f)(\delta(a))\bigr)
\\&=(\id\otimes f)\bigl((q\otimes\id)\circ\delta(a)\bigr)
\\&=0,\righttext{since $I\subset \ker (q\otimes\id)\circ\delta$.}
\end{align*}

Conversely, assume that $I$ is $B(G)$-invariant,
and let $a\in I$.
We need to show that $a\in \ker (q\otimes\id)\circ\delta$.
For every $f\in B(G)$ we have $f\cdot a\in I$, so
\begin{align*}
0
&=q(f\cdot a)
\\&=(\id\otimes f)\bigl((q\otimes\id)\circ\delta(a)\bigr).
\end{align*}
It follows that $(q\otimes\id)\circ\delta(a)=0$
since slicing by $B(G)$ separates points in $M(A\otimes C^*(G))$.
\end{proof}

\begin{rem}
It has been noticed elsewhere in the literature that the $B(G)$-module structure can be useful in other ways.
For example,
$\delta$ is slice-proper \cite[Definition~5.1]{exotic} if and only if the maps
\[
f\mapsto f\cdot a:B(G)\to A
\]
are weak*-weak continuous (for $a\in A$)
\cite[Lemma~5.3]{exotic}.
Also,
for any full coaction $(A,\delta)$,
\[
A_0:=\clspn\{A(G)\cdot A\}
\]
is a $C^*$-subalgebra and a nondegenerate $A(G)$-submodule of $A$,
where $A(G)$ is the Fourier algebra of $A$,
and $\delta$ is nondegenerate if and only if $A_0=A$
\cite[Lemma~1.2, Corollary~1.5]{fullred} (see also \cite[Lemma~2]{kat} 
In the same vein,
\cite[Corollary~1.7]{fullred} says that if $B$ is a nondegenerate $A(G)$-submodule of $M(A)$, then $\delta|_B$ is a nondegenerate coaction of $G$ on $B$.
\end{rem}

\end{appendix}

%\bibliographystyle{amsalpha}
%\bibliography{cstar}

\providecommand{\bysame}{\leavevmode\hbox to3em{\hrulefill}\thinspace}
\providecommand{\MR}{\relax\ifhmode\unskip\space\fi MR }
% \MRhref is called by the amsart/book/proc definition of \MR.
\providecommand{\MRhref}[2]{%
  \href{http://www.ams.org/mathscinet-getitem?mr=#1}{#2}
}
\providecommand{\href}[2]{#2}

\end{document}